\documentclass[12pt]{article}
\usepackage{latexsym, amsmath, amsfonts, amsthm, amssymb}
\usepackage{times}
\usepackage{a4wide}
\usepackage{times}
\usepackage{graphicx, tocloft}
\usepackage{mathrsfs}
\usepackage{url}
\usepackage{xcolor}
\usepackage{pgfplots}
\usepackage{amsmath}
\usepackage{hyperref}
\usepackage{mathtools}
\usepackage{dsfont}

\def \RR {\mathbb R}

\def \EE {\mathbb E}

\def \PP {\mathbb P}

\def \eps {\varepsilon}
\def \vphi {\varphi}

\newtheorem{theorem}{Theorem}[section]

\newtheorem{lemma}[theorem]{Lemma}

\newtheorem{proposition}[theorem]{Proposition}
\newtheorem{corollary}[theorem]{Corollary}
\newtheorem{claim}[theorem]{Claim}

\theoremstyle{definition}
\newtheorem{remark}[theorem]{Remark}

\def\myffrac#1#2 in #3{\raise 2.6pt\hbox{$#3 #1$}\mkern-1.5mu\raise 0.8pt\hbox{$
		#3/$}\mkern-1.1mu\lower 1.5pt\hbox{$#3 #2$}}

\def\qed{\hfill $\vcenter{\hrule height .3mm
		\hbox {\vrule width .3mm height 2.1mm \kern 2mm \vrule width .3mm
			height 2.1mm} \hrule height .3mm}$ \bigskip}

\def \cG {\mathcal G}

\def \cov {{ \rm Cov}\,}
\def \var {{ \rm Var}\,}

\begin{document}

\title{Long lines in subsets of large measure in high dimension}
\author{Dor Elboim and Bo'az Klartag}
\date{}

\maketitle

\begin{abstract}
    We show that for any set $A\subseteq [0,1]^n$ with $\text{Vol}(A)\ge 1/2$ there exists a line $\ell $ such that the one-dimensional Lebesgue measure of $\ell \cap A$ is at least $\Omega ( n^{1/4} )$. The exponent $1/4$ is tight. More generally, for a probability measure $\mu $ on $\mathbb R ^n$ and $0<a<1$ define
    \begin{equation*}
        L(\mu ,a):= \inf_{A ; \mu(A) = a} \sup _{\ell \text{ line}} \big|  \ell \cap A\big|
    \end{equation*}
    where $|\cdot | $ stands for the one-dimensional Lebesgue measure. We study the asymptotic behavior of $L(\mu ,a)$ when $\mu $ is a product measure and when $\mu $ is the uniform measure on the $\ell _p$ ball. We observe a rather unified behavior in a large class of product measures. On the other hand,
    for $\ell_p$ balls with $1 \leq p \leq \infty$ we find that there are phase transitions of different types.
\end{abstract}

\section{Introduction}

One of the simplest high-dimensional features of the geometry of $\RR^n$, for large $n$, is the fact that rather long  segments fit inside the $n$-dimensional unit cube. In fact, both the unit cube and the Euclidean ball of volume one contain segments of length $c \sqrt{n}$, for a universal constant $c > 0$. More generally, let $K \subseteq \RR^n$ be a convex body of volume one. The classical isodiametric inequality states that $K$ necessarily contains a segment of length $$ \bigg( \sqrt{\frac{2}{\pi e}} + o(1) \bigg) \cdot \sqrt{n}, $$
with the Euclidean ball being the extremal case. Can one avoid these long segments by restricting to a subset of $K$ of volume $1/2$? In order to exclude trivial answers involving removing a dense set of small measure, we slightly modify this question and formulate it precisely as follows: Does there exist a subset $A \subseteq K$ of volume $1/2$ such that for any line $\ell$ in $\RR^n$,
\begin{equation}  |A \cap \ell| < C \label{eq_1044} \end{equation}
for a universal constant $C > 0$? Here, $|A \cap \ell|$ is the one-dimensional length measure of the intersection of $A$ with the line $\ell$. We may answer this question in the affirmative in the example where $K$ is a Euclidean ball of volume one centered at the origin. In this case, the thin spherical shell
\begin{equation}  A = K \setminus \Big(  1 - \frac{1}{n} \, \Big) K \label{eq_1045} \end{equation}
has a volume of $1 - (1 - 1/n)^n \approx 1 - 1/e > 1/2$. An elementary argument based on the curvature of the sphere shows that this subset $A$ does not contain any long segment, and that (\ref{eq_1044}) holds true with a universal constant $C > 0$.
The answer is nearly affirmative, up to logarithmic factors,  also in the case of $\ell_p$-balls for $1 < p < 2$. That is, when
\begin{equation}  K =  B_p^n := \bigg\{  x\in \RR^n \, : \, \bigg( \sum_{i=1}^n |x_i|^p \bigg) ^{1/p} \leq \kappa_{p,n} \bigg\} \label{eq_1108} \end{equation}
where $\kappa_{p,n} = \Gamma(1 + n/p)^{1/n} / (2 \Gamma(1 +1/p))$ is chosen so that $K$ has volume one.  The situation changes  when one considers the case where $p \not \in (1,2]$, as we explain below. For a Borel probability measure $\mu$ on $\RR^n$ and a parameter $0 < a < 1$ we define
$$ L(\mu, a) := \inf_{A ; \mu(A) = a} \sup_{\ell \ \textrm{line }} |\ell \cap A| $$
where the infimum runs over all Borel subsets $A \subseteq \RR^n$ with $\mu(A) \geq a$, and the supremum runs over all lines $\ell \subseteq \RR^n$. We write $\lambda_K$ for the uniform probability measure on a convex body $K \subseteq \RR^n$ and abbreviate $L(K, a) = L(\lambda_K, a)$. The definition (\ref{eq_1108}) of $B_p^n$ makes sense for all $1 \leq p < \infty$, and by continuity   $B_{\infty}^n := \{ x \in \RR^n : \forall i, |x_i| \leq 1/2 \}$ is a unit cube.

\begin{theorem}\label{thm:1} Let $n \geq 1$ and $1 \leq p \leq \infty$. Then,
$$ L(B_p^n, 1/2) = \left \{ \begin{array}{lll} \Theta \left(n^{1/4} \right) & & p = 1, \infty \\ & & \\ \Theta_p \left((\log n )^{\frac{2-p}{2p}} \right) & & 1 < p \leq 2 \\ & & \\
\Theta_p \left(n^{\frac{p-2}{4p+2}} \right) & & 2 < p < \infty
\end{array}  \right. $$
Here, $\Theta(X)$ stands for a quantity $Y$ with $c X \leq Y \leq C X$ for universal constants $c, C > 0$. By $\Theta_p(X)$ we mean that the constants $c, C$ are not universal, but allowed to depend on $p$ solely.
\label{thm_1100}
\end{theorem}

\begin{remark}
The constant $1/2$ in the theorem can be replaced by any other fixed $a \in (0,1)$. However, the estimates will not be uniform as $a\to 0$ or $a\to 1$.
\end{remark}

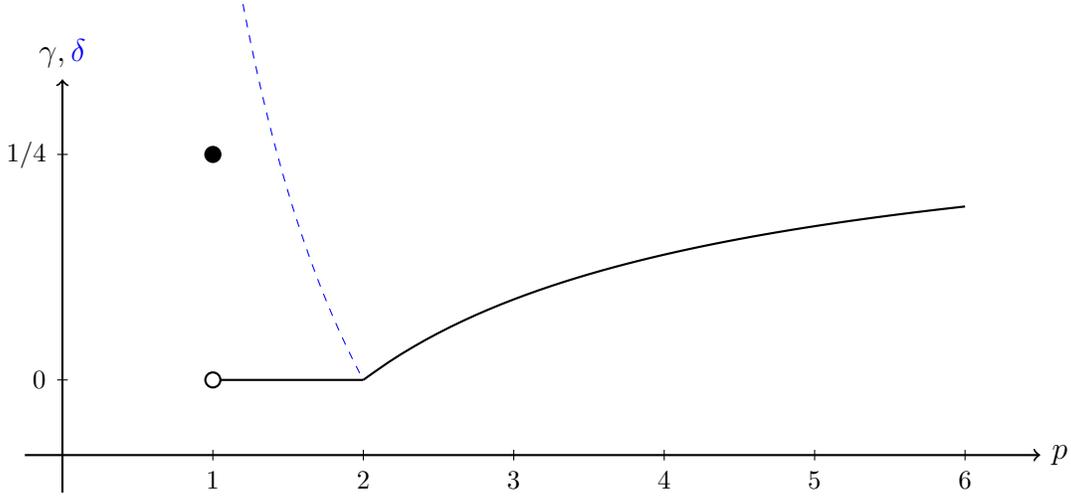
\begin{figure}
\begin{center}
\begin{tikzpicture}
  \draw[->,thick] (-0.5, 0) -- (13, 0) node[right] {$p$};
  \draw[->,thick] (0, -0.5) -- (0, 5) node[above] {$\gamma , {\color{blue}\delta }$};
  \draw[scale=1, thick, domain=2:6, smooth, variable=\p, black] plot ({2*\p}, {1+15*(\p-2)/(4*\p+2)});
  \draw[scale=1, domain=1.2:2, dashed, variable=\p, blue] plot ({2*\p}, {1+15*(2-\p)/(2*\p)});
  \draw[scale=1,thick, domain=1:2, smooth, variable=\p, black] plot ({2*\p}, {1});
  \draw[shift={(0,4)},color=black] (2pt,0pt) -- (-2pt,0pt) node[left] {\footnotesize $1/4$};
  \draw[shift={(0,1)},color=black] (2pt,0pt) -- (-2pt,0pt) node[left] {\footnotesize $0$};
  \draw[shift={(2,0)}, color=black] (0pt,2pt) -- (0pt,-2pt) node[below] {\footnotesize $1$};
  \draw[shift={(4,0)}, color=black] (0pt,2pt) -- (0pt,-2pt) node[below] {\footnotesize $2$};
  \draw[shift={(6,0)}, color=black] (0pt,2pt) -- (0pt,-2pt) node[below] {\footnotesize $3$};
  \draw[shift={(8,0)}, color=black] (0pt,2pt) -- (0pt,-2pt) node[below] {\footnotesize $4$};
  \draw[shift={(10,0)}, color=black] (0pt,2pt) -- (0pt,-2pt) node[below] {\footnotesize $5$};
  \draw[shift={(12,0)}, color=black] (0pt,2pt) -- (0pt,-2pt) node[below] {\footnotesize $6$};
  \draw[fill=white,thick]
        (2,1) circle (1mm) ;
        \draw[fill=black,thick]
        (2,4) circle (1mm) ;
\end{tikzpicture}
	\caption{The exponent $\gamma (p)$ appears in black. Theorem \ref{thm:1} states that $L(B_p^n,1/2)=\Theta_p (n^{\gamma (p) })$. The exponent $\delta (p) = (2-p) / (2p)$ in blue. Theorem \ref{thm:1} states that for $1<p<2$, $L(B_p^n,1/2)=\Theta_p  ((\log n)^{\delta (p) })$. }
	\label{fig:gammas figure}
	\end{center}
\end{figure}

A somewhat peculiar feature of Theorem \ref{thm_1100} is the exponent $\gamma(p)$ which equals $$ \frac{p-2}{4 p+2} $$ in the case $2 < p < \infty$,  and interpolates continuously between the values $0$ and $1/4$ attained at the endpoints $p = 2, \infty$. There is a discontinuity at $p=1$, where the exponent jumps to $1/4$.
For $p = 1, \infty$, an extremal set can be easily described, it suffices to consider the intrsection of $B_p^n$ with a {\it Euclidean spherical shell},
\begin{equation}  A = B_p^n \cap \left \{ x \in \RR^n \, : \, r \leq \| x \|_2 \leq r \left(1 + C / \sqrt{n} \right) \right \} \label{eq_1048} \end{equation}
for a certain value of $r = \Theta(\sqrt{n})$, where $\| x \|_p = \left( \sum_{i=1}^n |x_i|^p \right)^{1/p}$.
For $p$ in the range $(1, \infty)$, the Euclidean-norm considerations are less prominent in our construction of an extremal set $A$.

\medskip We move on to a detailed analysis of the case of the unit cube in $\RR^n$ with varying $a \in (0,1)$. In fact, our results hold not just for the uniform measure on the unit cube, but also for general product measures $\mu$ in $\RR^n$, satisfying the following conditions:

\begin{enumerate}
	\item[(i)] The measure $\mu$ is absolutely-continuous with density $\prod_{i=1}^n \rho_i(x_i)$, where the function $\rho_i: \RR \rightarrow [0, \infty)$ is smooth in the interval $(-1/2,1/2)$, and in this interval the derivatives $(\log \rho_i)^{(k)}$ for $k=0,\ldots,4$ are bounded in absolute value by $C$.
\item[(ii)] Sub-Gaussian tail: $\int_{-\infty}^{\infty} \exp(c t^2) \rho_i(t) dt \leq C$ for all $i$, for some constants $c, C > 0$.
\end{enumerate}

For example, the standard Gaussian measure in $\RR^n$, whose density is $(2 \pi)^{-n/2} \exp(- \| x \|_2^2 / 2)$,
satisfies (i) and (ii), as well as the uniform measure on the unit cube $B_{\infty}^n$.

\begin{theorem} Let $n \geq 1, 0 < a < 1$ and let $\mu$ be a probability measure on $\RR^n$ satisfying conditions (i) and (ii). Then,
$$ L(\mu, a) = \left \{ \begin{array}{lll} \Theta \left(a \cdot n^{1/4} \right) & & e^{-n} \leq a \leq 1/2 \\ & & \\ \Theta \left( n^{1/4} \cdot |\log (1 - a)|^{1/4} \right) & & 1/2 \leq a \leq 1- e^{-n}
\end{array}  \right. $$
Here, the implies constants in the $\Theta(..)$ notation
depend solely on the constants from conditions (i) and (ii).
\label{thm_1140}
\end{theorem}

Thus, unless $a$ is exponentially close to zero or to one, we observe {\it universality} in the class of product measures.
We can determine the value of $L(\mu, a)$ up to a constant factor, no matter what the precise details of the distribution $\rho_i$ are, as long as the tails are sub-Gaussian and the density is somewhat regular near the origin. We suspect that
the range $\min \{a,1-a \} <  e^{-n}$ corresponds to the ``large deviations'' regime where
the specifics of $\rho_i$ should matter.


\medskip It is possible to view Theorem \ref{thm_1100} and Theorem \ref{thm_1140} in the context of the Radon transform. Write $\cG$ for the collection of all lines in $\RR^n$. For a Borel measurable function $f: \RR^n \rightarrow \RR$ we define
$$ R f(\ell) = \int_{\ell} f. $$
When $f$ is, say, compactly-supported, the Radon transform $R f$ is a well-defined bounded function on $\cG$. Consider the case where $f = 1_A$, for a measurable subset $A \subseteq \RR^n$. Theorem \ref{thm_1100} and Theorem \ref{thm_1140} tell us that  $\sup R f$ is large, provided that $A$ has a substantial intersection with an $\ell_p$-ball, or that $A$ has a non-neglible mass with respect to a certain product measure $\mu$.

\medskip
In the discrete setting analogous questions have been studied, especially in the field of incidence geometry. In the discrete world, lines with large Lebesgue measure are replaced by lines having large number of points (large number of incidences on the line). For instance, see \cite[Theorem~3]{LS} for a bound on the number of $t$-rich lines with respect to a large subset of points (these are lines that contain at least $t$ points). This result is given in the general setting of block designs.

\medskip
Another discrete result that shares similarities with the problems considered in this paper, is the {\it density Hales-Jewett theorem}. This theorem states that in a sufficiently high dimension any subset of positive density contains a combinatorial line. The theorem was proved by  Furstenberg and Katznelson \cite{FK1,FK2}. See also \cite{DHJ,DKT} for combinatorial proofs. The exact statement is the following. For any $d\in \mathbb N$ and $\eps >0$ there exists $n_0=n_0(d,\eps )$ such that for all $n\ge n_0$ any subset $A\subseteq \{1,\dots ,d\}^n$ with $|A|\ge \eps d^n$ contains a combinatorial line. A combinatorial line is a set $\ell $ of size $d$ of the form
\begin{equation}
    \ell= \big\{ \big( \eps_1 a_1+(1-\eps _1)j,\dots , \eps_n a_n+(1-\eps _n)j  \big) \ : \ j\in \{1,\dots ,d\}\big\}
\end{equation}
where $a_1,\dots ,a_n\in \{1,\dots ,d\}$ and where $\eps _1,\dots ,\eps_n\in \{0,1\}$ are not all $1$. In other words, in a combinatorial line some coordinates (not all) are fixed, and some change linearly from $1$ to $d$.

\medskip The exponent $1/4$ observed in the case of the cube and the cross-polytope in Theorem \ref{thm_1100} is somewhat of a natural barrier for this problem. We say that a convex body $K \subseteq \RR^n$ of volume one is in {\it isotropic position} if its barycenter $b_K = \int_K x dx$ is at the origin and its covariance matrix
$$ Cov(\mu) = \int_K (x \otimes x) dx \in \RR^{n \times n} $$
is a scalar matrix. Here, $x \otimes x = (x_i x_j)_{i,j=1,\ldots,n} \in \RR^{n \times n}$. For example, the convex body $B_p^n$ is in isotropic position for all $p$ and $n$.

\begin{proposition} Let $n \geq 2$ and let $K \subseteq \RR^n$ be a convex body of volume one in isotropic position. Then,
$$ L(\lambda_K, 1/2) = O( n^{1/4} \log n ),  $$
where $O(X)$ stands for a quantity $Y$ with $Y \leq C X$ for a universal constant $C > 0$.
\label{prop_1127}
\end{proposition}
In order to prove Proposition \ref{prop_1127} we use a construction similar to (\ref{eq_1048}) above, and consider the intersection of $K$ with a thin Euclidean spherical shell,
\begin{equation}  A = K \cap \left \{ x \in \RR^n \, : \, r \leq \| x \|_2 \leq r \left(1 + C n^{-1/2} \sqrt{\log n} \right) \right \}, \label{eq_1050} \end{equation}
for a certain value of $r = O(\sqrt{n} \log n)$. Indeed, it follows from the recent works by Klartag and Lehec \cite{KLS Klartag,KLS Klartag Lehec} following 
the breakthrough by Chen \cite{Chen} and the  results by  Eldan and Klartag \cite{EK}, that the set $A$ captures at least $1/2$ of the mass of $K$. Yet the subset $A$ cannot intersect any line in a set whose length measure is above $n^{1/4} \log n$,
as we see from the proof of
Corollary \ref{cor:1} below.

\begin{remark} We remark in passing that for any convex body $K \subseteq \RR^n$ of volume one, we have
\begin{equation}
 L(\lambda_K, 1/2) \geq c, \label{eq_1242}
\end{equation}
for a universal constant $c > 0$. Our proof of (\ref{eq_1242}) uses convex geometric tools such as localization
and needle decomposition, and will be discussed elsewhere.
\end{remark}

\medskip In a vague sense that we were not able to make precise, we feel that the exponent $1/4$ corresponds
to the case of a ``generic'' convex body in isotropic position.
 Are there natural probability measures $\mu$ on $\RR^n$ for which
$L(\mu, 1/2)$ is much larger than $n^{1/4}$?
Such measures had better be unrelated to convexity and
without a product structure of the type considered above. The following proposition shows that there are measures $\mu $ such that the coordinates of a random vector $X$ with law $\mu $ are typically of order $1$ but still $L(\mu ,1/2)$ is much larger than $n^{1/4}$.

\begin{proposition} Let $X, Y$ be independent, standard Gaussian random vectors in $\RR^n$. Let $U$ be a random variable, independent of $X$ and $Y$, that is distributed uniformly in the interval $[0,1]$. Write $\mu$ for the probability measure on $\RR^n$ that is the law of the random vector
	$$ X + U Y. $$
Then
$$ L(\mu, 1/2) = \Theta( \sqrt{n} ). $$
\label{prop_1151}
\end{proposition}


\medskip Unless stated otherwise, throughout this text we use the letters $c, C, \tilde{C}$ etc. to denote positive universal constants, whose value may change from one
line to the next. 

\subsection{Main ideas in the proofs}

One may think of our main technique to prove lower bounds on $L(\mu ,a)$ as a Mermin-Wagner type argument in a random direction,
    or alternatively, as an approximate needle decomposition into segments that are as long as possible.
In order to explain this technique we sketch the proof of Theorem~\ref{thm_1140} in the special case where $a=1/2$ and $\mu =\gamma _n$ is the standard Gaussian measure
in $\RR^n$.

\medskip
Suppose that $A \subseteq \RR^n$ satisfies $\gamma_n(A) \geq 1/2$. We would like to prove that there exists a line $\ell \subseteq \RR^n$ with
\begin{equation}  |A \cap \ell| \geq c n^{1/4} \label{eq_924} \end{equation}
where $|\cdot|$ stands for the one-dimensional Lebesgue measure. Let $Z, W$ be independent standard Gaussian random vectors in $\RR^n$. It is well-known (e.g., \cite[Chapter 2]{V})
that
\begin{equation} \PP \left(|W| > \sqrt{n}/2 \right) > 0.9. \label{eq_937} \end{equation}
Furthermore we claim that there exists $c_1>0$ such that
for any $r \leq c_1 n^{-1/4}$,
\begin{equation}  d_{TV}(Z, Z + r W) < 0.1 \label{eq_926} \end{equation}
where $d_{TV}(X,Y) = \sup_{B} \left| \PP(X \in B) - \PP(Y \in B) \right|$ is the total variation distance between $X$ and $Y$.
A neat proof of (\ref{eq_926}) using Pinsker's inequality can be found in \cite[Theorem~1.1]{DMR}. It follows from \eqref{eq_926} and the fact that $\mathbb P (Z\in A)=\gamma _n(A) \ge 1/2$ that
\begin{equation*}
    \mathbb P \big( Z+rW\in A \big) \ge 0.4.
\end{equation*}
Since the last inequality holds for any $r\le c_1n^{-1/4}$, it holds when replacing $r$ with a random variable $U$ distributed uniformly in the interval $[0,c_1n^{-1/4}]$ that is independent of $Z$ and $W$. We obtain using \eqref{eq_937} that
\begin{equation}
    \mathbb P \big( Z+UW\in A, \  |W|\ge \sqrt{n}/2 \big) \ge 0.3.
\end{equation}
It follows that there are realizations $z,w\in \mathbb R ^n$ with $|w|\ge \sqrt{n}/2$ such that
\begin{equation} \mathbb P \big( z+Uw\in A \big)\ge 0.3. \label{eq_1157} \end{equation}
Finally, note that the last probability is exactly the normalized one-dimensional Lebesgue measure of the intersection of $A$ with the line segment $[z,z+c_1n^{-1/4}w]$. This line segment is of length $c_1n^{-1/4} |w|\ge cn^{1/4}$, completing the proof of (\ref{eq_924}).

\medskip We may now explain the proof of Proposition \ref{prop_1151}. By the definition of $\mu$, we
have $z,w \in \RR^n$ with $|w| \geq \sqrt{n} / 2$ such that (\ref{eq_1157}) holds, where now $U$ is uniformly distributed in
the interval $[0,1]$. This proves the lower bound for $L(\mu, 1/2)$. The upper bound follows
by considering the set $A$ which is a Euclidean ball of radius $5 \sqrt{n}$ centered at the origin in $\RR^n$.

\medskip What we see from the above is that in order to obtain lower bounds for $L(\mu ,a)$ it is useful to ``push'' the measure $\mu $ in a random direction. Equation~\eqref{eq_926} shows that in the Gaussian case one can push the measure to a distance of order $n^{1/4}$ without changing it by much in total variation. In the mathematical physics literature, this technique of pushing a measure was introduced in \cite{MW} and is usually referred to as a Mermin-Wagner type argument.
In the convexity literature \cite{Eldan, GM, KLS, K} it is quite common to decompose a measure on an $n$-dimensional space into
one-dimensional needles, as in the approximate decomposition into uniform measures on segments discussed above.

\medskip
Consider next the case of the uniform measure on the unit cube $[0,1]^n$. Let $X=(X_1,\dots ,X_n)$ be a uniform point in the cube and note that the coordinates $X_i$ are i.i.d. uniform random variables in $[0,1]$. In order to prove Theorem~\ref{thm_1140} in this case, the first attempt would be to use the same perturbation as in the Gaussian case. That is, to perturb each coordinate to $Y_i:=X_i+n^{-1/4}Z_i$ where $Z_1,\dots ,Z_n$ are i.i.d. $\!$normal random variables. However, it is easy to see that such a perturbation will push the random point outside of the unit cube with high probability and the total variation distance $d_{TV}(X,Y)$ will tend to $1$. To overcome this issue we only perturb coordinates which are not too close to $0$ or $1$. More precisely, we use a perturbation of the form
\begin{equation}\label{eq:478}
    Y_i:=X_i+n^{-1/4}\varphi (X_i)Z_i,
\end{equation}
where $\varphi $ is a smooth bump function supported on $[1/3,2/3]$. We  show that for a suitable choice of $\varphi$, the perturbation in \eqref{eq:478} satisfies $d_{TV}(X,Y)\le 0.1$. This strategy can be used to obtain lower bounds on $L(\mu ,a)$ for general product measures $\mu $ as long as $a$ does not tend to $0$ or $1$. When $a$ is small, this strategy fails as the set $A$ can be concentrated around the center of the cube where the density of $Y=(Y_1,\dots ,Y_n)$ can be very small. To obtain tight bounds in this case, instead of perturbing the original measure, we first tilt the measure slightly toward the center of the cube and then perturb it randomly. See Section~\ref{sec_aa} for more details.

\medskip
In Section~\ref{sec:p} we study the case of $\ell _p$ balls. The idea here is to perturb the coordinates as much as possible without changing the $\ell_p$ norm by much. It turns out that when $p>2$ it is better to only perturb coordinates which are close to zero, of order $n^{-1/(2p+1)}$ while for $1<p<2$ one should only perturb large coordinates of order $\log ^{1/p}n$. Another difference between the two regimes is that when $p>2$ we perturb each coordinate independently like in \eqref{eq:478} but for $1<p<2$ such a perturbation will change the $p$ norm by too much. In order to handle this issue we perturb pairs of consecutive coordinates at a time. For each pair, we perturb the first coordinate of the pair randomly and then use the other coordinate of the pair to ``correct'' the $\ell_p$ norm.

\medskip
We turn to briefly explain how to obtain the upper bounds on $L(\mu ,A)$. We wish to find a set $A$ with large measure, for which $|\ell \cap A|$ is small for all lines $\ell $. In the case of product measures, when $a>1/2$, we choose $A$ to be a Euclidean spherical shell with appropriate width and use concentration of measure. When $a<1/2$ we use the super-additivity property of $L(\mu ,a)$, namely $L(\mu,a+b)\ge L(\mu,a)+L(\mu ,b)$ (see Lemma~\ref{lem:super} below). In the case of $\ell_p$ balls, when $1<p<2$ we choose $A$ to be a $p$-spherical shell instead of a Euclidean spherical shell. When $2<p<\infty $ we take our spherical shell to be in some sense Euclidean for small coordinates and $\ell _p$ for larger coordinates (see Lemma~\ref{lem:9374} for more details).

\subsection{Extensions and open problems}

There are a few natural extensions of the results in this paper. Perhaps some of those can solved using the methods developed in our proofs.
One such extension is to estimate $L(B_p^n,1/2)$ uniformly in $p$. For example, as $p$ tends to $1$, at what rate does the behavior change from logarithmic to $n^{1/4}$?
It would also be interesting to understand the asymptotic behavior of $L(B_p^n,a)$ as $a\to 0$ or $a\to 1$ like in Theorem~\ref{thm_1140}.

\medskip Another question, is what can be said when the lines in our main theorems are replaced by higher dimensional planes or by other low degree polynomial curves.
Perhaps the most interesting problem which we were not able to solve is to understand $L(K,1/2)$ for a general convex body $K$. For example, is there a simple geometric parameter of $K$  that explains the asymptotic behavior of $L(K,1/2)$?

\subsection{Acknowledgements} We thank Ron Peled for suggesting an earlier version of the problem as part of a joint work on first passage percolation \cite{Dor}. We thank
Ronen Eldan for interesting discussions on needle decompositions. We thank Barbara Dembin for the Euclidean spherical shell construction. We thank Jonathan Zung for the useful idea with the discrete second derivative (see \eqref{eq:5} for example). We thank Noga Alon and Zeev Dvir for telling us about the analogous discrete problems. We thank Arnon Chor, Paul Dario, Sa'ar Zehavi and Sahar Diskin for fruitful discussion. Finally, we thank the anonymous referee for her or his thoughtful comments. BK was partially supported by a grant from the Israel Science Foundation (ISF).

\section{Product measures}\label{sec:product}

In this section we prove Theorem \ref{thm_1140}. We start with the proof of the lower bound for small $a$.

\subsection{Lower bound for general product measures and $e^{-n} \leq a \leq 1/2$}
\label{sec_aa}
Let $\mu$ be a measure satisfying requirement (i) from Section 1.
The constants hidden in the $O(...)$ notation in this proof may depend on $C$ from requirement (i).
Let $n \geq 2$ and fix real numbers $R, r$ with $ |R|, |r| \leq 1$.

\medskip Fix a $C^4$-smooth, non-negative bump
function $\vphi$ on the real line, supported in $(-1/3,1/3)$ with $\vphi(0) = 1/100$ and $|\vphi'| < 1$.
For concreteness, say that
$\vphi(t) = (1 - 9 t^2)^5 / 100$ for $|t| < 1/3$ and $\vphi(t)$ vanishes for $|t| \geq 1/3$. Define
\begin{equation}  g_i = \frac{\left( \vphi^2 \rho_i \right)''}{\rho_i},   \label{eq_248} \end{equation}
which is $C^2$-smooth in the real line and supported in $(-1/2, 1/2)$ thanks to requirement (i).
Let $X_1,\ldots,X_n$ be independent random variables, where
the density of $X_i$ is
\begin{equation} \tilde{\rho}_i(t) =  \kappa_{i,R}  e^{-R^2 g_i(t)} \rho_i(t). \label{eq_925}\end{equation}  Here,
\begin{equation} \kappa_{i,R} = \left( \int_{-\infty}^{\infty} e^{-R^2 g_i(t)} \rho_i(t) dt \right)^{-1} =  1 + O(R^4),
\label{eq_347} \end{equation}
where in the last equality we Taylor expanded the exponential and used the fact that $\int g_i \rho_i =\int (\varphi ^2 \rho _i)''= 0$ which follows as $\varphi ^2 \rho _i$ is a $C^4$-smooth bump function. Note that the implied constant in the $O(R^4)$ depends on $C$ from (i).
Let $\delta_1,\ldots,\delta_n \in \{ -1, 1 \}$ be independent symmetric Bernoulli variables, independent of the $X_i$'s.
Let $U \in [0,1]$ be a uniform random variable, independent of all of the previous random variables. Consider the perturbation
$$ Y_i = X_i + r U \delta_i \vphi(X_i). $$
\begin{proposition} The density $f = f_{R,r}$ of the random vector $Y = (Y_1,\ldots,Y_n)$ satisfies
$$ f(y) \geq  e^{-\tilde{C} (R^4 n + r^4 n)} \cdot e^{\left(  r^2 / 6- R^2 \right) \cdot  \sum_{i=1}^n g_i(y)  } \cdot \prod_{i=1}^n \rho_i(y_i), $$
where $\tilde{C} > 0$ depends solely on the constant from requirement (i). \label{prop_1120}
\end{proposition}

In order to prove Proposition \ref{prop_1120} we denote
$$ \tilde{Y}_i = X_i + r \delta_i \vphi(X_i). $$
Observe that the two maps $t \mapsto t \pm r \vphi(t)$ are monotone increasing in $\RR$ as  $r \leq 1$.
We may therefore apply the change of variables formula, and conclude that the density of the random variable $\tilde{Y}_i$ equals
\begin{equation}  f_i(t) = f_i^{(R,r)}(t) = \frac{1}{2} \left[ \frac{\tilde{\rho}_i(x_1)}{1 + r \vphi'(x_1)} + \frac{ \tilde{\rho}_i(x_2) }{1 - r \vphi'(x_2)}  \right] \label{lem_347} \end{equation}
where $x_1 < t < x_2$ are determined by $x_1 + r \vphi(x_1) = t = x_2 - r \vphi(x_2)$.
We emphasize that $f_i(t)$ as defined in (\ref{lem_347}) depends also on the parameters $r$ and $R$. Its Taylor approximation
with respect to these two parameters, which is uniform in $t$, is given in the following lemma.

\begin{lemma} For any $t \in \RR$ and $i=1,\ldots,n$,
$$ f_i(t) =  \exp \left[ \left( r^2 / 2 - R^2 \right) \cdot  g_i(t) + O \left(r^4 + R^4 \right) \right] \cdot \rho_i(t), $$
where the implicit constant in $O(r^4 + R^4)$ depends only on $C$ from requirement (i) of Section 1.
\label{lem_1056}
\end{lemma}

\begin{proof}
Since $|r| \leq 1$, the equation $x_1 + r \vphi(x_1) = t$ implies
$$ x_1 = t - r \vphi(t) + r^2 \vphi'(t) \vphi(t) + O(|r|^3) $$
and similarly
$$ x_2 = t + r \vphi(t) + r^2 \vphi'(t) \vphi(t) + O(|r|^3). $$
Thus, from (\ref{eq_925}), (\ref{eq_347}) and (\ref{lem_347}),
\begin{align*} f_i(t) & =  \frac{1}{2} \cdot  \frac{ 1 - R^2 g_i(t) + R^2 r g_i^{\prime}(t) \vphi(t)  + O(R^4 + R^2 r^2) }{1 + r \vphi'(t) - r^2 \vphi''(t) \vphi(t) + O(|r|^3)} \cdot \rho_i(x_1)  \\ & + \frac{1}{2} \cdot  \frac{ 1 - R^2 g_i(t) - R^2 r g_i^{\prime}(t) \vphi(t) + O(R^4 + R^2 r^2) }{1 - r \vphi'(t) - r^2 \vphi''(t) \vphi(t) + O(|r|^3)} \cdot \rho_i(x_2). \end{align*}
Next,
\begin{equation}
 \frac{\rho_i(x_1)}{\rho_i(t)} =  1 - r \vphi (\log \rho_i)' + r^2 (\log \rho_i)' \vphi' \vphi
+ r^2 \vphi^2 \frac{(\log \rho_i)'' + [(\log \rho_i)']^2 }{2} + O(|r|^3),
\label{eq_549} \end{equation}
where $\vphi, \rho$ and their derivatives are evaluated at $t$. The expression for $\rho_i(x_2) / \rho_i(t)$ is similar
to the right-hand side of (\ref{eq_549}), the only difference is that the coefficient of
$r \vphi (\log \rho_i)'$ is $+1$ and not $-1$. Consequently,
\begin{align} \nonumber \frac{f_i(t)}{\rho_i(t)} & = 1 - R^2 g_i  + r^2 \left[ \vphi'' \vphi + (\vphi')^2 + 2 (\log \rho_i)' \vphi' \vphi
+ \vphi^2 \frac{(\log \rho_i)'' + [(\log \rho_i)']^2 }{2} \right] + O(R^4 + |r|^3)
\\ & = 1 - R^2 g_i  + r^2 \frac{(\vphi^2 \rho_i)''}{2 \rho_i}  + O(R^4 + |r|^3). \label{eq_1033} \end{align}
The function $f_i(t) / \rho_i(t)$ is an even function of $r$ and of $R$, hence the Taylor expansion in $r$ and $R$ contains only even powers of $r$ and $R$.
Thus one might expect to improve $O(R^4 + |r|^3)$ to $O(R^4 + r^4)$ in (\ref{eq_1033}). Indeed, since $\vphi$ and $\log \rho_i$ have bounded $C^4$-norm, the odd terms in the Taylor expansion vanish as the function is even, and the error in the Taylor approximation is $O(R^4 + r^4)$,
uniformly in $t$.
To summarize,  for any $t \in \RR$,
$$ \frac{f_i(t)}{\rho_i(t)}  = 1 + \left(\frac{r^2}{2} - R^2 \right) g_i(t)   + O(R^4 + r^4).
$$
Since $\sup |g_i| < C$ and $|r|, |R| \leq 1$, the lemma follows.
\end{proof}

\begin{proof}[Proof of Proposition \ref{prop_1120}]
From the definition of $f$, we have $f(y) = \EE_U \prod_{i=1}^n f_i^{(R, Ur)}(y_i)$. Thus, by Lemma \ref{lem_1056},
\begin{equation}  f(y) =  \EE_U \exp \left[ O \left(r^4 n + R^4 n \right) + \left( U^2 r^2 / 2 - R^2 \right) \cdot \sum_{i=1}^n g_i(y_i) \right] \cdot \prod_{i=1}^n \rho_i(y_i). \label{eq_227} \end{equation}
From Jensen's inequality,
$$ f(y) \geq  e^{-\tilde{C} \left( r^4 n + R^4 n \right)} \cdot \exp \left[ \EE_U \left( U^2 r^2 / 2 - R^2 \right) \cdot \sum_{i=1}^n g_i(y_i) \right] \cdot \prod_{i=1}^n \rho_i(y_i). $$
Since $\EE U^2 = 1/3$, the conclusion of the proposition follows.
\end{proof}

\begin{proposition}
Let $n \geq 2$ and let $\mu$ be a probability measure on $\RR^n$ satisfying requirement (i) from Section 1. Then there
exists $\tilde{c} > 0$, depending solely on the constant from  (i), such that if $e^{-n} \leq a \leq 1/2$ then,
$$ L(\mu, a) \geq \tilde{c} \cdot n^{1/4} \cdot a. $$
\label{prop_919}
\end{proposition}

\begin{proof} Set $r = n^{-1/4} / (2 \tilde{C})^{1/4}$ with $\tilde{C}$ from
Proposition \ref{prop_1120}. We will consider a mixture of two distributions.
Write $Y^{(1)}$ for a random vector with density $f^{(1)} := f_{0, r}$, it has the law of the random vector $Y$ with the parameter $R = 0$.
Let $Y^{(2)}$ be a random vector with density $ f^{(2)} := f_{r, r}$. It has the law of the random vector $Y$ with the parameter  $R = n^{-1/4} / (2 \tilde{C})^{1/4} $.
By Proposition \ref{prop_1120},
$$ f^{(1)}(y) \geq  \frac{1}{e}  \cdot e^{\frac{1}{6} \cdot  \frac{\sum_{i=1}^n g_i(y)}{\sqrt{2 \tilde{C} n}}  } \cdot \prod_{i=1}^n \rho_i(y_i), $$
while
$$ f^{(2)}(y) \geq   \frac{1}{e} \cdot e^{-\frac{5}{6} \cdot  \frac{\sum_{i=1}^n g_i(y)}{\sqrt{2 \tilde{C} n}}  } \cdot \prod_{i=1}^n \rho_i(y_i). $$
Consequently,
\begin{equation}  \frac{f^{(1)}(y) + f^{(2)}(y)}{2} \geq \frac{1}{2 e} \cdot \prod_{i=1}^n \rho_i(y_i) \qquad \qquad \qquad \text{for all} \ y \in \RR^n. \label{eq_1057}
\end{equation}
Let $A \subseteq \RR^n$ satisfy $\mu(A) \geq a \geq e^{-n}$. According to (\ref{eq_1057}) either for  $i=1$ or for $i=2$,
\begin{equation}  \int_A f^{(i)} \geq \frac{1}{2e} \cdot a. \label{eq_1100} \end{equation}
Let $R = 0$ in case $i=1$ and $R = r$ in case $i = 2$.
Let $X$ be distributed as above with the parameter $R$, i.e.,
$X_1,\ldots, X_n$ are independent with density given in (\ref{eq_925}). Denote
$$ Z_i =  \delta_i \vphi(X_i). $$
Then the random vector $Y = (Y_1,\ldots,Y_n)$ defined via $ Y_i = X_i + r U \delta_i \vphi(X_i) $ satisfies
$$ Y = X + r U Z. $$
There exists some $c_1 > 0$ depending on the constant from condition (i) such that
\begin{equation} \PP \big( |Z|^2 \le  c_1^2  n \big) \le \mathbb P \big(  |\{i: Z_i^2\ge c_1\}|\le c_1n\big) \le \frac{1}{20} \cdot e^{-n}, \label{eq_1121}
\end{equation}
where in the last inequality we used standard estimates on the Binomial distribution and the fact that $\mathbb P (Z_i^2>c_1)\ge 0.99$ as long as $c_1$ is sufficiently small.
The random vector $Y$ has density $f^{(i)}$.  Inequality (\ref{eq_1100}) thus means that
\begin{equation}  \PP(X + r U Z \in A) \geq \frac{1}{2e} \cdot a. \label{eq_1130} \end{equation}
Since $a \geq e^{-n}$, from (\ref{eq_1121}) and (\ref{eq_1130})
we deduce that there exist $x, z \in \RR^n$ with $|z| > c_1 \sqrt{n}$ such that
$$ \PP(x + r U z \in A) \geq c' a. $$
This means that
\begin{equation}
 \frac{|A \cap [x, x + r z]|}{r |z|} \geq c' a. \label{eq_153} \end{equation}
Since $|z| > c_1 \sqrt{n}$ and $r > \tilde{c} n^{-1/4}$, the segment in (\ref{eq_153})
is of length at least $\tilde{c} n^{1/4}$, completing the proof.
\end{proof}

\subsection{Lower bound for $1/2 \leq a \leq 1-e^{-n}$}
\label{sec_bb}

We continue with the notation and assumptions of Section \ref{sec_aa}. We use the parameter value $R = 0$, while $r$ will be determined soon.
In particular  $X_1,\ldots,X_n$ are independent random variables, where $\rho_i$ is the law of $X_i$, and
$$ Y_i = X_i + r U \delta_i \vphi(X_i). $$
From (\ref{eq_227}) we know that the density $f$ of $Y$ satisfies
\begin{equation}  f(y) \leq e^{\tilde{C} r^4 n} \cdot \EE_U \exp \left[ \frac{U^2 r^2}{2}  \cdot \sum_{i=1}^n g_i(y_i)  \right]
\cdot \prod_{i=1}^n \rho_i(y_i), \label{eq_359} \end{equation}
with $\tilde{C}$ depending on the constant from requirement (i). Denote $W_i = g_i(X_i)$,
where $g_i$ is defined in (\ref{eq_248}). The random variables $W_1,\ldots, W_n$
are independent and have mean zero since $\mathbb E [W_i]= \int  g_i\rho _i = \int   (\varphi ^2 \rho _i)''=0$ . From requirement (i) and from (\ref{eq_248}), we see that for all $i$,
$$ |W_i| < C'  $$
for some  $C'$ depending solely on the parameter from requirement (i). Therefore, by Hoeffding’s inequality  $\sum_{i=1}^n W_i / \sqrt{n}$ is a sub-Gaussian
random variable, in the terminology of \cite[Section 2.5]{V}. That is,
$$ \PP \left( \left| \frac{\sum_{i=1}^n W_i}{\sqrt{n}} \right| \geq t \right) \leq \tilde{C} \exp(- \tilde{c} t^2) \qquad \qquad \qquad (t \in \RR) $$
for $\tilde{c}, \tilde{C} > 0$ depending on the parameter from requirement (i).

\begin{lemma} Let $\mathcal A $ be an event with $\mathbb P (\mathcal A)=\epsilon  \le 1/2$ and let $W$ be a sub-Gaussian random variable. Then for any $0 < s < \sqrt{|\log \eps|} $,
\begin{equation}  \EE  \mathds 1_{\mathcal A}e^{sW}   \leq C_1 \cdot \eps \cdot e^{C_2 s \sqrt{|\log \eps|}},
\label{eq_313} \end{equation}
where $C_1, C_2 >0$ depend solely on the sub-Gaussianity constants of $W$, i.e., on the constants $\tilde{C}, \tilde{c} > 0$ such that 
$\PP(|W| \geq t) \leq \tilde{C} e^{-\tilde{c} t^2}$ for all $t \in \RR$. 
\label{lem_353}
\end{lemma}

\begin{proof} The left-hand side of (\ref{eq_313}) equals
\begin{align*}
\int_0^{\infty} & \PP \left( \mathcal A , e^{sW} \geq t \right) dt
\leq \int_0^{\infty} \min \left \{ \eps, \tilde{C} \exp \left( -\tilde{c} \frac{\log^2 t}{s^2} \right) \right \} dt \\ &
\leq  \eps \cdot e^{\hat{C} s \sqrt{|\log \eps|}} + \tilde{C} \int_{e^{\hat{C} s \sqrt{|\log \eps|}}}^{\infty}
\exp \left( -\tilde{c} \frac{\log^2 t}{s^2} \right)  dt  \\ & =
 \eps \cdot e^{\hat{C} s \sqrt{|\log \eps|}} + \tilde{C} s \int_{\hat{C} \sqrt{|\log \eps|}}^{\infty}
\exp \left( s r -\tilde{c} r^2 \right)  dr \leq C_1 \eps \cdot e^{C_2 s \sqrt{|\log \eps|}},
\end{align*}
provided that we choose $\hat{C}$ large enough.
\end{proof}

\begin{proposition}
Let $n \geq 1$ and let $\mu$ be a probability measure on $\RR^n$ satisfying condition (i) from Section 1. Then there
exists $\tilde{c} > 0$ depending solely on the constant from  (i), such that if $1/2 \leq a \leq 1 - e^{-n}$ then for $\eps = 1 - a$,
$$ L(\mu, a) \geq \tilde{c} n^{1/4} |\log \eps|^{1/4}. $$ \label{prop_1726}
\end{proposition}

\begin{proof} Let $A \subseteq \RR^n$ satisfy $\mu(A) = 1 - \eps$. As mentioned before, the random variable $W:=\sum _{i=1}^n W_i/\sqrt{n}$ is sub-Gaussian. Thus, applying Lemma \ref{lem_353} for the event $\mathcal A :=\{X\notin A\}$ and the random variable $W$ we obtain
\begin{align} \nonumber \EE_U \int_{A^c}  & \exp \left[ \frac{U^2 r^2}{2}  \cdot \sum_{i=1}^n g_i(y_i)  \right] \prod_{i=1}^n \rho_i(y_i) dy
= \EE  \mathds 1_{ \{ X \not \in A \}} \exp \left[ \frac{\sqrt{n} U^2 r^2}{2} \cdot \frac{\sum_{i=1}^n W_i}{\sqrt{n}}
\right] \\ & \leq C_1 \EE _U  \eps \cdot e^{C_2 \sqrt{n} U^2 r^2 \sqrt{|\log \eps|}} \leq C_1  \eps \cdot e^{C_2 \sqrt{n} r^2 \sqrt{|\log \eps|}} \leq \sqrt{\eps} \label{eq_1143}
\end{align}
provided that $r = c n^{-1/4} |\log \eps|^{1/4}$ for a small enough constant $c$ (depending solely on the parameter from requirement (i)). We select $c$ small enough so that $c^4 \tilde{C} < 1/4$ where $\tilde{C}$ is the constant from (\ref{eq_359}).
From (\ref{eq_359}) and (\ref{eq_1143}) we conclude that
\begin{equation}  \PP( Y \not \in A) \leq \sqrt{\eps} \cdot e^{\tilde{C} r^4 n} \leq \sqrt{\eps} \cdot \eps^{-1/4} = \eps^{3/4} \leq 9/10. \label{eq_401} \end{equation}
However, $Y = X + r U Z$, where as above $ Z_i =  \delta_i \vphi(X_i) $ and by (\ref{eq_1121}),
\begin{equation} \PP( |Z| > \hat{c}_1 \sqrt{n}) > 19/20. \label{eq_947} \end{equation}
Consequently,
from (\ref{eq_401}) and (\ref{eq_947}) there exist $y, z \in \RR^n$ with $|z| > \hat{c}_1 \sqrt{n}$ such that
$$ \PP( x + r U z \in A ) \geq 1/20. $$
Hence at least a $0.05$-fraction of the points in the segment $[x, x + r z]$ belong to $A$, and the length of this segment is
at least $c n^{1/4} |\log \eps|^{1/4}$.
\end{proof}

\subsection{Upper bounds for $0 < a \leq 1 - e^{-n}$}

Let $\mu$ be a probability measure on $\RR^n$ satisfying condition (i) and (ii) from Section 1. The constants $c, C$
in this section depend solely on those from conditions (i) and (ii).
Let $X = (X_1,\ldots,X_n)$ be a random vector with law $\mu$. It follows from condition (i) that
$$ E_i := \EE X_i^2 > c \qquad \qquad \qquad \text{for all} \ i. $$
 On the other hand, condition (ii) states that $X_i$ is sub-Gaussian, and hence $X_i^2 - E_i$ is sub-exponential,
in the terminology of \cite[Section 2.8]{V}. Denote
\begin{equation}  E = \sqrt{\EE |X|^2} \in (c \sqrt{n}, C \sqrt{n} ). \label{eq_420} \end{equation} From Bernstein's inequality \cite[Theorem 2.8.1]{V},
for $t > 0$,
\begin{equation} \PP \left (  \left| \frac{|X|^2 - E^2}{\sqrt{n}} \right| \geq t \right) \leq \tilde{C} e^{-\tilde{c} \min \{ t^2, t \sqrt{n} \}}. \label{eq_408}
\end{equation}
Since $| |X| - E | \leq | |X|^2 - E^2 | / E \leq C| |X|^2 - E^2 | / \sqrt{n}$, we conclude from (\ref{eq_408}) that
\begin{equation} \PP \left ( E - t \leq  |X| \leq E + t \right) \geq 1 - \tilde{C} e^{-\tilde{c} \min \{ t^2, t \sqrt{n} \} }.
\label{eq_412}
\end{equation}

\begin{lemma}  For any $1/2 \leq a \leq 1 - e^{- n}$ we have $ L(\mu, a) \leq \tilde{C} \cdot n^{1/4} \cdot |\log (1-a)|^{1/4}$.
\label{lem_316}
\end{lemma}

\begin{proof} Define $\eps = 1-a$ and set
$$ A = \left \{ x \in \RR^n \, : \, E - \hat{C} \sqrt{|\log \eps|} \leq |x| \leq E + \hat{C} \sqrt{|\log \eps|} \right \} $$
so that $\mu(A) \geq 1- \eps$ thanks to (\ref{eq_412}). We need to show that $|A \cap \ell|$
is at most $C n^{1/4} |\log \eps|^{1/4}$ for any line $\ell$. Indeed, it follows from (\ref{eq_420}) that for any $x \in A$,
$$ E^2 - \tilde{C} \sqrt{n \cdot |\log \eps|} \leq |x|^2 \leq E^2 + \tilde{C}  \sqrt{n \cdot |\log \eps|}. $$
In particular, if $\ell(t) = x + t y$ with $x,y \in \RR^n$ and $|y| = 1$, the set of $t \in \RR$ for which $\ell(t) \in A$
is contained in the set of $t \in \RR$ for which $t^2$ belongs to an interval of length at most $2\tilde{C}  \sqrt{n \cdot |\log \eps|}$.
This set is of Lebesgue measure at most $C' n^{1/4} |\log \eps|^{1/4}$, completing the proof.
\end{proof}

Lemma \ref{lem_316} proves the upper bound in Theorem \ref{thm_1140} in the range $1/2 \leq a \leq 1 - e^{-n}$.
An upper bound for the range $a \in [0,1/2]$ will be obtained from the upper bound in the case $a = 1/2$ and the following
super-additivity property:

\begin{lemma}\label{lem:super}
Let $\mu$ be an absolutely continuous measure in $\RR^n$. Then, for any $a,b \in (0,1)$ with $a + b < 1$ we have
$$ L(\mu, a+b) \geq L(\mu, a) + L(\mu, b). $$
\label{lem_536}
\end{lemma}

Let $L(A) := \sup_{\ell} |A \cap \ell|$ and note that $L(\mu ,a)=\inf _{\mu (A)\ge a} L(A)$. For the proof of the lemma we need the following claims.

\begin{claim}\label{claim:con}
The function $a \mapsto L(\mu, a)$ is monotone and continuous in $(0,1)$.
\end{claim}

\begin{proof}
The monotonicity of $L(\mu, a)$ is clear. Observe that $L(A \setminus B(x, \eps)) \geq L(A) - 2 \eps$. Let $\eps > 0$ and $a \in (0,1)$.
Then there is $\delta > 0$ such that for each set $A \subseteq \RR^n$ with $\mu(A) > a/2$ there is $x \in \RR^n$ with $\mu(A \cap B(x, \eps)) > \delta$.
By considering a near contender for the infimum of $L(\mu, a+\delta/2)$ and removing from it a ball of radius $\eps$ we obtain
$$ L(\mu, a- \delta/2) \geq L(\mu,a + \delta/2) - 2 \eps. $$
This proves the continuity at $a$.
\end{proof}

\begin{claim}\label{claim:cube}
For any $0<\lambda <1$ and $\epsilon >0$ there is a set $B \subseteq [0,1]^n$ with $\text{Vol}(B)\ge \lambda $ such that for any line $\ell $ we have $|\ell \cap B| \le \lambda |\ell \cap [0,1]^n|+\epsilon$.
\end{claim}

\begin{proof}
We write $Q:=[1,2]^n$ and (for technical reasons) prove the statement for this cube, rather than for $[0,1]^n$. For $r> 0$ denote $Q_r = Q \cap r S^{n-1}$ and observe that the sets $Q_r$ are disjoint and are subsets of spheres. Next, we subdivide $[0, \infty )$ into sufficiently small intervals and pick roughly $\lambda$-fraction of these small intervals that are roughly uniformly distributed in $[0, \infty )$. More precisely, let $\delta =\delta (\epsilon ,n)$ be sufficiently small (to be determined later) and let $I_j:=[\delta (j-1),\delta j )$. Let $k:=\lfloor \delta ^{-1/5} \rfloor $ and for all $i\le k$ let
$$B_i:=\bigcup _{m= 0}^{\infty } \bigcup_{ j=1}^{\lceil \lambda k \rceil} \bigcup _{r\in I_{mk+i+j}} Q_r.$$
Since every point in $Q$ is covered by at least $\lceil \lambda k \rceil $ of the sets $\{B_i\}_{i=1}^k$, it follows that there is some $i_0\le k$ for which $\text{Vol}(B_{i_0})\ge \lambda $. We will show that $B:=B_{i_0}$ satisfies the inequality
\begin{equation}\label{eq:39564}
 |\ell \cap B| \le \lambda |\ell \cap Q|+\epsilon
\end{equation}
for any line $\ell$. If $|\ell \cap Q|\le \epsilon $ then \eqref{eq:39564} trivially holds. Thus, we may assume that $|\ell \cap Q|\ge \epsilon $. For $j\ge 1$ we let $a_j$ be the length of the intersection of $\ell $ with the part of the spherical shell  $\bigcup _{r\in I_j}Q_r$. We claim that for all $j\ge 1$ we have $a_j \le 4\sqrt{\delta} n^{1/4}$. Indeed, it follows as the width of these shells is $\delta $ and the radii of the relevant shells are at most $2\sqrt{n}$. Next, let $j_0$ be the first integer for which $a_{j_0}>0$. We claim that starting from $j_0+1$, the sequence $a_j$ is monotonically decreasing. This follows by a simple geometric consideration using facts that $\ell \cap Q$ is a line segment and that the spherical shells have the same width.

Letting $m_0:=\lceil j_0/k \rceil $ We obtain
\begin{equation}
\begin{split}
    \lceil \lambda k \rceil |\ell \cap Q| \ge \lceil \lambda k \rceil \sum _{j=j_0}^\infty a_j \ge \lceil \lambda k \rceil \sum _{m=m_0}^{\infty } \sum _{j=1}^k a_{mk+j+i_0} \ge k \lceil \lambda k \rceil  \sum _{m=m_0}^{\infty }  a_{(m+1)k+i_0}  \\
    \ge k \sum _{m=m_0}^{\infty } \sum _{j=1}^{\lceil \lambda k \rceil} a_{(m+1)k+j+i_0} \ge k|\ell \cap B| -4k^2 \sqrt{\delta }n^{1/4} ,
\end{split}
\end{equation}
where in the third and fourth inequalities we used the monotonicity of $a_j$ and in the last inequality we used the definition of $B$ and the bound on $a_j$. This finishes the proof of \eqref{eq:39564} as long as $\delta $ is sufficiently small depending on $n$ and $\epsilon $.
\end{proof}

We can now prove Lemma~\ref{lem:super}.

\begin{proof}[Proof of Lemma~\ref{lem:super}]
Since any set $A \subseteq \RR^n$ contains a compact $K$ with $\mu(A \setminus K) < \eps$,
we conclude from Claim~\ref{claim:con} that
$$ L(\mu, a) = \inf_{\mu(K) = a} L(K) $$
where the infimum runs over all compacts $K \subseteq \RR^n$ with $\mu(K) = a$. Next, write $K_{\delta}$ for the $\delta$-neighborhood
of the compact $K$, and observe that $L(K_{\delta}) \longrightarrow L(K)$ as $\delta \rightarrow 0^+$. We say that $A \subseteq \RR^n$ is elementary
if it is a finite unions of cubes, each of the form $Q = \prod_{i=1}^n [a_i, b_i)$. For any compact $K$ and $\delta > 0$
we may find an elementary set contained in $K_{\delta}$ and containining $K$. It follows that
\begin{equation}  L(\mu, a) = \inf_{\mu(A) = a} L(A) \label{eq_046} \end{equation}
where the infimum runs over all elementary sets $A \subseteq \RR^n$ with $\mu(A) = a$.

Next, let $a,b \in (0,1)$ satisfy $a + b < 1$ and denote $\lambda = a / (a+b)$. Let $\epsilon >0$ and let $A$ be an elementary set with $\mu (A)=a+b$ such that $L(A)\le L(\mu ,a+b) +\epsilon $. It follows from Claim~\ref{claim:cube} that there is a set $B\subseteq A$ with $\text{Vol}(B)=\lambda \text{Vol}(A)=a$ such that for any line $\ell $ we have $|\ell \cap B|\le \lambda | \ell \cap A|+\epsilon $. In particular we have $L(B)\le \lambda L(A)+\epsilon $. We obtain that
\begin{equation}
    L(\mu ,a)\le L(B) \le \lambda L(A)+\epsilon \le \lambda L(\mu ,a+b) +2\epsilon .
\end{equation}
Similarly we have that $L(\mu ,b)\le  (1-\lambda ) L(\mu ,a+b) +2\epsilon$ and therefore
\begin{equation}
L(\mu ,a)+L(\mu ,b)\le L(\mu ,a+b)+4\epsilon .
\end{equation}
This finishes the proof of the lemma.
\end{proof}

We immediately obtain the following corollary

\begin{corollary} For any $0 < a \leq 1/2$ we have $ L(\mu, a) \leq \tilde{C} n^{1/4} \cdot a$.
\label{cor_541}
\end{corollary}

\begin{proof}
    By Lemma~\ref{lem_316} and by monotonicity and super-additivity of $a\mapsto L(\mu ,a)$ we obtain that for all $0<a\le 1/2$ we have
    \begin{equation*}
        Cn^{1/4} \ge L(\mu ,1/2) \ge L(\mu ,a\cdot \lfloor 1/(2a)\rfloor ) \ge \lfloor 1/(2a) \rfloor  \cdot L(\mu ,a) \ge L(\mu ,a) /(4a),
    \end{equation*}
    as needed.
\end{proof}

The proof of Theorem \ref{thm_1140} is now complete, as the lower bounds follow from Proposition
\ref{prop_919} and Proposition \ref{prop_1726}, while the upper bounds follow from Lemma \ref{lem_316} and Corollary \ref{cor_541}.

\begin{remark} The sub-Gaussian assumption in condition (ii) is not really used in the proof of Corollary \ref{cor_541},
and it may be replaced by weaker conditions such as $\int_{-\infty}^{\infty} t^4 \rho_i(t) dt \leq C$.
\end{remark}

\section{The case of $\ell _p$ balls}\label{sec:p}

In this section we prove Theorem~\ref{thm:1}. Throughout this section none of the estimates will be uniform in $p$. Thus, the constants $C,c,C_0,c_0$ as well as the $O$ and $\Theta$ notations are allowed to depend on $p$. In the proof we will use the following result from \cite[Theorem 1]{Naor}. Recall the definition of $B_p^n$ given in \eqref{eq_1108} and note that by Stirling's approximation $\kappa _{p,n}:=\Gamma(1 + n/p)^{1/n} / (2 \Gamma(1 +1/p))$ is of order $\Theta (n^{1/p})$.

\begin{theorem}\label{thm:asaf}
Let $p>0$ and $n\ge 1$. Let $g_1,\dots ,g_n$ be i.i.d. random variables with density
\begin{equation*}
\frac{1}{2\Gamma (1+1/p)}e^{-|t|^p}, \quad t\in \mathbb R
\end{equation*}
 and let $Z$ be an independent $\exp (1)$ random variable. Then, the random vector
\begin{equation}\label{eq:27}
    X=(X_1,\dots ,X_n) :=\frac{\kappa_{p,n}  }{\big( \sum _{i=1}^n |g_i|^p +Z \big)^{1/p}}(g_1,\dots ,g_n)
\end{equation}
is uniformly distributed in $B_p^n$.
\end{theorem}

The following claims quantify the fact that the coordinates of a uniform point in $B_p^n$ are roughly independent and behave like a constant multiple of the random variables $g_i$ given in Theorem~\ref{thm:asaf}. To state the claims we let
\begin{equation}\label{eq: def of tilde X}
    \tilde{X}_i:=a_ng_i \quad \text{where} \quad a_n:=\frac{\kappa _{p,n} p^{1/p}}{n^{1/p}}=\frac{e^{-1/p}}{2\Gamma (1+1/p)} \Big(1+O\Big( \frac{\log n}{n} \Big) \Big),
\end{equation}
where the last equality follows from the definition of $\kappa _{p,n }$ after equation \eqref{eq_1108} and from Stirling's formula. Intuitively, the random variable multiplying $(g_1,\dots ,g_n)$ in \eqref{eq:27} is concentrated around $a_n$ and therefore $X_i$ can be approximated by $\tilde{X}_i$. In the claims we let $X$ be the uniform point in $B_p^n$ given by \eqref{eq:27}. By symmetry, it suffices to consider the first two coordinates $X_1$ and $X_2$ in the claims below. The first claim is Corollary~1 in \cite{ABP}.

\begin{claim}\label{claim:cov}
We have that $\text{Cov} \big( X_1 ^2,X_2^2 \big) \le 0$.
\end{claim}

\begin{claim}\label{claim:X and tilde X}
We have that $\mathbb E \big[ (X_1-\tilde{X}_1)^2 \big] \le C/n$
\end{claim}

\begin{claim}\label{claim:independence phi}
Let $\varphi $ be a compactly supported differentiable function such that $\varphi '$ is a Lipschitz function. Then, there exist a constant $C>0$ depending on $\varphi $ such that for all $n\ge 1$ and  $1\le R \le \sqrt{n}$ we have
\begin{enumerate}
\item
    \begin{equation*}
        \mathbb E \big[ \varphi (RX_1)-\varphi (R\tilde{X}_1) \big] \le \frac{C}{n}
    \end{equation*}
    \item
    \begin{equation*}
        \text{Cov}\big( \varphi (RX_1),\varphi (RX_2)  \big) \le \frac{C}{nR}
    \end{equation*}
\end{enumerate}
\end{claim}

The proofs of Claim~\ref{claim:independence phi} and Claim~\ref{claim:X and tilde X} are slightly technical and we postpone the proofs to Appendix~\ref{sec:A}.

\subsection{Upper bounds}

The next corollary follows from Theorem~5 in \cite{ABP}.

\begin{corollary}\label{cor:l2 computation}
Let $n\ge 1$,  $p \in [1,\infty ] $ and let $X$ be a uniform sample from $B_p^n$. Then,
\begin{equation*}
    \text{Var} \big( ||X||_2^2 \big) \le Cn.
\end{equation*}
\end{corollary}

\begin{proof}
By Claim~\ref{claim:cov} we have
\begin{equation*}
    \var \big( ||X||_2^2 \big) =\var \Big( \sum _{i=1}^n X_i^2 \Big) \le \sum _{i=1}^n \var \big( X_i^2 \big)  \le Cn
\end{equation*}
as claimed.
\end{proof}

\begin{remark}
In fact a more careful analysis shows that when $p\neq 2$,
\begin{equation*}
    \var \big( ||X||_2^2 \big) =(1+o(1))\frac{p\Gamma (5/p)\Gamma (1/p)-(p+4)\Gamma (3/p)^2 }{p\Gamma (1/p)^2} n.
\end{equation*}
\end{remark}

We can now prove the following corollary that gives the right upper bound in the case that $p$ is $1$ or $\infty $.

\begin{corollary}\label{cor:1}
For all $p\in [1,\infty ]$ we have $L(B_p^n,1/2) \le C n^{1/4} $.
\end{corollary}

\begin{proof}
By Corollary~\ref{cor:l2 computation} and Chebyshev's inequality, there exists $C_0>0$ such that
\begin{equation*}
    \mathbb P \big(  \big| |X|^2 -\mathbb E [|X|^2] \big| \ge C_0\sqrt{n}  \big) \le 1/2
\end{equation*}
and therefore the set
\begin{equation*}
    A:= \big\{  x\in B_p^n : \big|  |x|^2-\mathbb E [|X|^2] \big| \le C_0\sqrt{n}   \big\}
\end{equation*}
has volume at least $1/2$. We claim the any line $\ell $ satisfies $|\ell \cap A | = O( n^{1/4})$. To this end, note that a line can intersect $A$ in at most two intervals. We claim that each of these intervals has length of at most $O(n^{1/4})$. Indeed, let $x$ and $x+y$ be the endpoints of one of these intervals and consider the function
\begin{equation*}
    f(t):=|x+ty|^2=\sum _{i=1}^n (x_i+ty_i)^2.
\end{equation*}
For all $0\le t \le 1 $ we have that $x+ty\in A$ and therefore $f(t)=\mathbb E [|X|^2]+O(\sqrt{n})$. Thus
\begin{equation*}
    \frac{1}{2}\sum _{i=1}^n y_i^2 = f(1)+f(0)-2f(1/2) \le C\sqrt{n}.
\end{equation*}
and therefore the length of this interval is $|y|\le Cn^{1/4}$.
\end{proof}

The proof of the following corollary is similar to the proof of Corollary~\ref{cor:1} but it uses an $\ell _p$ spherical shell instead of an $\ell _2$ spherical shell.

\begin{corollary}\label{cor:2}
For all $1<p\le 2$ we have that $L(B_p^n,1/2) \le  C (\log n )^{\frac{2-p}{2p}} $.
\end{corollary}

\begin{proof}
Define the sets
\begin{equation*}
    A:=\big\{x\in B_p^n : ||x||_p^p \ge \kappa _{p,n}^p-C_0 \big\} \quad \text{and} \quad B:= \big\{ \forall i \le n ,\ |x_i|\le C_0 \log ^{1/p}n \big\}.
\end{equation*}
where $C_0>0$ is a sufficiently large constant that will be chosen later. We start by proving that $\text{Vol}(A\cap B)\ge 1/2$. We have
$$ \text{Vol} \big( \big\{ x\in \mathbb R^n: \| x \|_p^p \leq \kappa_{p,n}^p - C_0 \big\} \big) = \left( 1 - \frac{C_0}{\kappa_{p,n}^p} \right)^{n/p} \le 1/4,
$$
where the last inequality holds as long as $C_0$ is sufficiently large since  $\kappa _{p,n}=\Theta _p(n^{1/p})$. It follows that $\text{Vol}(A)\ge 3/4$.

We turn to bound the volume of $B$. Let $X$ be the uniform point in $B_p^n$ given in Theorem~\ref{thm:asaf}. By Bernstein's inequality and the fact that the random variables $|g_i|^p$ from Theorem~\ref{thm:asaf} have exponential tails, we have that
\begin{equation}\label{eq:777}
    \mathbb P \Big( c_1 n \le \sum _{i=1}^n |g_i|^p \le C_1n \Big) \ge 1 - C e^{-cn}.
\end{equation}
for some $c_1,C_1>0$. Moreover, the density of $g_i$ is proportional to $e^{-|t|^p}$ and therefore $\mathbb P \big( |g_i|\ge 2\log ^{1/p}n \big)\le Cn^{-2} $. Thus, by \eqref{eq:777} and Theorem~\ref{thm:asaf}, as long as $C_0$ is sufficiently large we have that $\mathbb P \big( |X_i|\ge C_0\log ^{1/p}n \big)\le C n^{-2} $. It follows from a union bound that $\text{Vol}(B)\ge 1-C/n$ and therefore $\text{Vol}(A\cap B )\ge 1/2$.

We turn to show that for any line $\ell $, we have that $|\ell \cap A \cap B | =O\big( (\log n)^{\frac{2-p}{2p}} \big)$. The set $A$ is the difference of two convex sets and therefore the intersection $\ell \cap A$ is the union of at most two intervals. As in the proof of Corollary~\ref{cor:1}, it suffices to bound the length of the intersection of each of these intervals with $B$. Let $x$ and  $x+y$ be two points inside one of these intervals such that $x\in B$. It suffices to show that $|y|=O\big( (\log n)^{\frac{2-p}{2p}} \big)$. To this end, define the functions
\begin{equation*}
f_i(t):=(x_i+ty_i)^p \quad \text{and} \quad f(t):=||x+ty||_p^p=\sum _{i=1}^n f_i(t).
\end{equation*}
Since the interval $[x,x+y]$ is contained in $A$ we have that $\kappa _{p,n}^p -C_0 \le f(t) \le \kappa _{p,n}^p$ for $t\in \{0,1/2,1\}$ and therefore
\begin{equation}\label{eq:5}
    4C_0 \ge f(0)+f(1)-2f(1/2) =\sum _{i=1}^n f_i(0)+f_i(1)-2f_i(1/2).
\end{equation}
Next, we have that
\begin{equation}\label{eq:7}
\begin{split}
    f_i(0)+f_i(1)-&2f_i(1/2)= \int _0^{1/2}sf_i''(s)ds+\int _{1/2}^1 (1-s)f_i''(s)ds \\
    &\ge \frac{1}{4}\int _{1/4}^{1/2} p(p-1)y_i^2(x_i+sy_i)^{p-2} ds \ge c y_i^2 \big( \max ( |y_i|,|x_i| ) \big)^{p-2} \\
    &\ge c \min \big( |y_i|^p, |x_i|^{p-2}y_i^2 \big)\ge c\min \big( |y_i|^p, ( \log n )^{\frac{p-2}{p}} y_i^2 \big),
\end{split}
\end{equation}
where the first equality holds for any function and the last inequality follows as $x\in B$. We claim that $\min \big( |y_i|^p, ( \log n )^{\frac{p-2}{p}} y_i^2 \big) \ge c( \log n )^{\frac{p-2}{p}} y_i^2 $. Indeed, if this minimum is $|y_i|^p$ then by \eqref{eq:7} and \eqref{eq:5} we have that $|y_i|\le C$ and therefore $|y_i|^p\ge cy_i^2 \ge c( \log n )^{\frac{p-2}{p}} y_i^2$. Thus, we get the bound
\begin{equation*}
    f_i(0)+f_i(1)-2f_i(1/2) \ge c( \log n )^{\frac{p-2}{p}} y_i^2.
\end{equation*}
Substituting this bound into \eqref{eq:5} we get that
\begin{equation*}
    |y|^2 =\sum _{i=1}^n y_i^2 \le C (\log n)^{\frac{2-p}{p}}.
\end{equation*}
This finishes the proof of the corollary.
\end{proof}

In the last two corollaries we saw that the Euclidean spherical shell and the $\ell _p$ spherical shell can be used to obtain upper bounds on $L(B_p^n ,1/2)$. The main idea of the proof of the following lemma is to consider a set which looks like a Euclidean shell for coordinates close to $0$ and like an $\ell _p$ shell for larger coordinates.
\begin{lemma}\label{lem:9374}
For all $2<p<\infty $ we have that $L(B_p^n,1/2)\le C  n^{\frac{p-2}{4p+2}} $.
\end{lemma}

\begin{proof}
Define the convex function
\begin{equation*}
    h(r):= \left \{ \begin{array}{lll} \frac{p}{2} n^{\frac{2-p}{2p+1}}r^2 +\big(1-\frac{p}{2} \big)n^{- \frac{p}{2p+1}} & & |r|\le n^{-\frac{1}{2p+1} } \\ & & \\ \quad \quad \quad |r|^p & & |r|\ge n^{-\frac{1}{2p+1} }
\end{array}  \right.
\end{equation*}
Next, let $E:=n\mathbb E [h(X_i)]$ and consider the set
\begin{equation*}
    A:= \Big\{ x\in B_p^n : \ \Big| \sum _{i=1}^n h(x_i) -E\Big| \le C_0 \Big\},
\end{equation*}
where $C_0$ is a sufficiently large constant that will be determined later. We start by proving that $\text{Vol}(A)\ge 1/2$. To this end, let  $g(r)=h(r)-|r|^p$ and define the sets
\begin{equation*}
    A_1:=\big\{ x\in B_p^n : \ \big| ||x||_p^p -E_1 \big| \le C_0/2 \big\}
\end{equation*}
and
\begin{equation*}
A_2:=\Big\{ x\in B_p^n : \ \Big| \sum _{i=1}^n g(x_i)   -E_2 \Big| \le C_0/2 \Big\},
\end{equation*}
where $E_1:=n\mathbb E \big[ |X_i|^p \big]$ and $E_2:=n\mathbb E \big[g(X_i) \big]$. It suffices to lower bound the volumes of $A_1$ and $A_2$ since $A_1\cap A_2 \subseteq A$. By the same arguments as in the proof of Corollary~\ref{cor:2} we have that $\text{Vol}(A_1)\ge 3/4$ as long as $C_0$ is sufficiently large.

In order to lower bound the volume of $A_2$ we estimate the variance of $\sum g(X_i)$. To this end, let
\begin{equation*}
    \varphi (t):= \mathds 1 \{ |t|\le  1\} \cdot \big( \frac{p}{2} t^2+1-\frac{p}{2} -|t|^p \big),\quad t\in \mathbb R
\end{equation*}
and note that $\varphi $ is differentiable and $\varphi '$ is Lipschitz. For all $r\in \mathbb R $ we have that
\begin{equation*}
    g(r)=n^{-\frac{p}{2p+1}} \varphi \big( n^{\frac{1}{2p+1}} r \big)
\end{equation*}
and therefore, by the second part of Claim~\ref{claim:independence phi} we have that
\begin{equation*}
\begin{split}
    \text{Var} &\Big( \sum_{i=1}^n g(X_i)  \Big)= n^{-\frac{2p}{2p+1}}  \cdot \text{Var} \Big( \sum_{i=1}^n \varphi \big( n^{\frac{1}{2p+1}}X_i \big)  \Big) \\
    &= n^{-\frac{2p}{2p+1}}  \sum_{i=1}^n \text{Var} \Big(  \varphi \big( n^{\frac{1}{2p+1}}X_i \big)  \Big)+n^{-\frac{2p}{2p+1}}  \sum_{i\neq j} \text{Cov} \Big(  \varphi \big( n^{\frac{1}{2p+1}}X_i \big),\varphi \big( n^{\frac{1}{2p+1}}X_j \big)  \Big) \le C.
\end{split}
\end{equation*}
Thus, as long as $C_0$ is sufficiently large, we have by Chebyshev's inequality $\text{Vol}(A_2)\ge 3/4$ and therefore $\text{Vol}(A)\ge 1/2$.

We turn to show that $|\ell \cap A|\le Cn^{\frac{p-2}{4p+2}}$ for any line $\ell $. This part of the proof is identical to the corresponding part in the proof of Corollary~\ref{cor:2} and therefore some of the details are omitted. Let $x,y\in \mathbb R ^n$ such that the line segment from $x$ to $x+y$ is contained in $A$. It suffices to show that $|y|\le Cn^{\frac{p-2}{4p+2}}$.  We have that $h''(r)\ge c n^{-\frac{p-2}{2p+1}}$ for all $r$ except for two points where $h'$ is not differentiable and therefore the function $f(t):=\sum _{i=1}^n h(x_i + t y_i)$ satisfies
\begin{equation*}
    C\ge f(0)+f(1)-2f(1/2) \ge c n^{-\frac{p-2}{2p+1}} \sum _{i=1}^n y_i^2,
\end{equation*}
where in the first inequality we used that the line segment from $x$ to $x+y$ is contained in $A$. This finishes the proof of the lemma.
\end{proof}

\subsection{Lower bound when $2\le p<\infty $}

The main result in this section is the following proposition.

\begin{proposition}\label{prop:lower bound p>2}
 For all $2\le p<\infty $ we have that $L(B_p^n,1/2)=\Omega _p\big( n^{\frac{p-2}{4p+2}}  \big) $.
\end{proposition}

The main idea of the proof is to use a perturbation that changes only coordinates close to $0$. Recall that $X$ is uniform random variable in $B_p^n$. Let $\psi $ be a smooth non-negative bump function supported on $[1,2]$. We think of $\psi $ as a fixed function and allow the constants $C$ and $c$ to depend on $\psi $. Let $R,r>0$ and let $\varphi (x):=r\psi (Rx)$. Finally, let $\delta_1,\ldots,\delta_n$ be i.i.d. $\!$symmetric $\{-1,1\}$ Bernoulli random variables. Define the random vector $Y:=(Y_1,\dots ,Y_n)$ by
\begin{equation}\label{eq:perturbation}
Y_i:=X_i+\varphi(X_i)\delta _i,
\end{equation}
where $X_i$ are given in \eqref{eq:27}. The main idea of the proof of Proposition~\ref{prop:lower bound p>2} is the following proposition that shows that the perturbation of $X$ given in \eqref{eq:perturbation} does not change the distribution of $X$ by much.

\begin{proposition}\label{prop:total variation p}
For all $2<p<\infty $ there exists a small constant $\eps >0$ such that the following holds. Let $n\ge 1$ be sufficiently large depending on $\eps $ and let $1\le R\le \sqrt{n}$, $r\le 1$ be such that
\begin{equation}\label{eq:assumption}
     nR^3r^4 \le \eps ^6 ,\quad R^{2p+1}\ge n.
\end{equation}
Finally, let $Y$ be the random variable defined by \eqref{eq:perturbation}. Then,
\begin{equation*}
    d_{TV}( X,Y )\le 1/4.
\end{equation*}
\end{proposition}

Using Proposition~\ref{prop:total variation p} we can easily prove Proposition~\ref{prop:lower bound p>2}.

\begin{proof}[Proof of Proposition~\ref{prop:lower bound p>2}]
    Let $n\ge 1$ be sufficiently large and $R:=n^{1/(2p+1)}$. Let $X$ be a uniform point in $B_p^n$ and define the random vector $W=(W_1,\dots ,W_n)$ where $W_i:=\psi (RX_i)\delta _i$.

    We start by showing that $|W|^2 =\sum _{i=1}^n \psi (RX_i)^2$ is typically large. Recall the definition of $\tilde{X}_i$ in \eqref{eq: def of tilde X}. We clearly have that
    \begin{equation*}
        \mathbb E \big[ \psi (R\tilde{X}_i)^2\big] \ge c/R ,\quad \mathbb E \big[ \psi (R\tilde{X}_i)^4 \big] \le C/R
    \end{equation*}
    and therefore, by the first part of Claim~\ref{claim:independence phi} we have
    \begin{equation*}
        \mathbb E \big[ \psi (R{X}_i)^2\big] \ge c/R ,\quad \mathbb E \big[ \psi (R{X}_i)^4\big] \le C/R.
    \end{equation*}
    It follows that $\mathbb E [|W|^2]\ge cn/R$ and moreover, using the second part of Claim~\ref{claim:independence phi} we have
    \begin{equation*}
\begin{split}
    \text{Var} \big( |W|^2 \big) &\le  \sum _{i=1}^n  \mathbb E \big[ \psi (R{X}_i)^4 \big] +\sum _{i\neq j} \text{Cov} \big( \psi (RX_i)^2,\psi (RX_i) ^2 \big) \le \frac{Cn}{R}.
\end{split}
\end{equation*}
    Thus, by Chebyshev's inequality there exists some $c_1\ge 0$ such that
    \begin{equation}\label{eq:W is large}
        \mathbb P \big( |W|\ge c_1\sqrt{n/R} \big) \ge 0.99,
    \end{equation}
    as long as $n$ is sufficiently large. Next, fix $\eps >0$ such that the conclusion of Proposition~\ref{prop:total variation p} hold and let $r_0:=\eps ^2R^{-3/4}n^{-1/4}$. Observe that, by Proposition~\ref{prop:total variation p}, for any $r\le r_0$ we have
    \begin{equation*}
        d_{TV}(X,X+rW)\le 1/4.
    \end{equation*}
    It follows that for all $A\subseteq B_p^n$ with $\text{Vol} (A)\ge 1/2$ we have that
    \begin{equation}\label{eq:prob of A}
        \mathbb P \big( X+rW\in A \big) \ge 1/4.
    \end{equation}
    Since \eqref{eq:prob of A} holds for all $r\le r_0$ it also holds when replacing $r$ with a random variable $U\sim U[0,r_0]$ that is independent of $X$ and $W$. Thus, by \eqref{eq:W is large} there are some realizations $x\in B_p^n$ and $w\in \mathbb R ^n$ with $|w|\ge c_1\sqrt{n/R}$ such that
    \begin{equation*}
        \mathbb P \big( x+Uw\in A \big) \ge 1/5.
    \end{equation*}
   The last probability is exactly the normalized Lebesgue measure of the intersection of $A$ with the line segment $[x,x+r_0w]$. Thus, letting $\ell $ be the line containing $x$ and $x+r_0w$ we obtain
    \begin{equation*}
        |\ell \cap A |\ge |r_0w|/5 \ge cr_0\sqrt{n/R} \ge   c_\eps n^{1/4} R^{-5/4} = c_{\eps }n^{\frac{p-2}{4p+2}}
    \end{equation*}
    as needed.
\end{proof}

The rest of this section is devoted to the proof of Proposition~\ref{prop:total variation p}. Throughout the proof we assume that $\eps $ is sufficiently small and $n$ sufficiently large depending on $\eps $. We start with the following lemma that gives a closed form expression to the density of $Y$.

\begin{lemma}\label{lem:density}
The density of the random vector $Y=(Y_1,\dots ,Y_n)$ defined in \eqref{eq:perturbation} is given  by
\begin{equation}\label{eq:density p}
    f(y)=\mathbb E \Big[ \mathds 1 \big\{ x(y,\delta)\in B_p^n \big\} \cdot \prod _{i=1}^n \big( 1+\varphi '(x_i)\delta _i \big)^{-1} \Big],\quad y\in \mathbb R ^n,
\end{equation}
where $x(y,\delta):=(x_1,\dots ,x_n)$ and $x_i=x_i(y_i,\delta _i)$ is the random variable defined to be the solution of the equation $y_i=x_i+\varphi (x_i)\delta _i $.
\end{lemma}

\begin{proof} Assuming that $rR<1/2$ (which follows from \eqref{eq:assumption} and $R\le \sqrt{n}$), and $z\in \{-1,1\}$ the map $t \mapsto t+\varphi (t)z$ is a diffeomorphism and it has the Jacobian $1+\varphi '(t)z>0$. Thus, by the change of variables formula, the density at $y$ conditioning on $\delta =(\delta _1,\dots ,\delta _n)$ is given by
\begin{equation*}
    f^{(\delta )}(y)=\mathds 1 \big\{ x(y,\delta)\in B_p^n \big\} \cdot \prod _{i=1}^n \big( 1+\varphi '(x_i)\delta _i \big)^{-1}.
\end{equation*}
It follows that the unconditional density is given by
\begin{equation*}
    f(y)=\mathbb E \Big[ \mathds 1 \{x(y,\delta)\in B_p^n \} \cdot \prod _{i=1}^n \big( 1+\varphi '(x_i)\delta _i \big)^{-1} \Big].
\end{equation*}
This finishes the proof of the lemma.
\end{proof}

In order to estimate the density given in Lemma~\ref{lem:density} we restrict our attention to a subset of $B_p^n$ of almost full measure. To this end, for $y\in \mathbb R ^n$ let $I(y):=\{ i : 1\le Ry_i\le 2\}$  and for $t\in \mathbb R$ let $g(t):=\varphi ''(t)\varphi (t)+\varphi '(t)^2=(\varphi ^2/2 )''(t)$. Consider the set $A=A_1\cap A_2\cap A_3$ where
\begin{equation*}
    A_1:=\big\{ y\in B_p^n: ||y||_p^p \le \kappa_{p,n}^p - \eps  \big\},\quad A_2:= \Big\{ y:  |I(y)|\le \frac{n}{R\eps } \Big\}
\end{equation*}
and
\begin{equation*}
    A_3:=\Big\{ y : \Big| \sum _{i=1}^n g(y_i) \Big|  \le \eps \  \Big\}.
\end{equation*}

Proposition~\ref{prop:total variation p} follows immediately from the following two lemmas.

\begin{lemma}\label{lem:13}
We have that $\text{Vol} (A) \ge 1-C\eps $.
\end{lemma}

\begin{lemma}\label{lem:density is close to 1 p>2}
For all $y\in A$ we have that $|f(y)-1|\le C\eps $.
\end{lemma}

\begin{proof}[Proof of Proposition~\ref{prop:total variation p}]
By Lemma~\ref{lem:density is close to 1 p>2} we have $
    \big| \mathbb P \big( Y\in A \big) -\mathbb  P \big( X \in A \big) \big|  \le C\eps $
and by Lemma~\ref{lem:13} we have $\mathbb P (X\notin A )\le C\eps $. It follows that $\mathbb P (Y\notin A) \le C\eps $ and therefore, using Lemma~\ref{lem:density is close to 1 p>2} once again we obtain
 \begin{equation*}
 \begin{split}
     d_{TV}(X,Y)&=\frac{1}{2} \int _{\mathbb R ^n} \big| f(y)-\mathds 1 \{y\in B_p^n\} \big| dy \\
     &\le \frac{1}{2} \int _A \big| f(y)-1 \big|dy +\mathbb P (X\notin A) +\mathbb P (Y\notin A )\le C\eps .
 \end{split}
 \end{equation*}
 This finishes the proof of the proposition as long as $\eps $ is sufficiently small.
\end{proof}

It remains to prove Lemma~\ref{lem:13} and Lemma~\ref{lem:density is close to 1 p>2}.

\begin{proof}[Proof of Lemma~\ref{lem:13}]
The first part of this proof is similar to the proof of Corollary~\ref{cor:2}. Let $X$ be the uniform point in $B_p^n$ given by \eqref{eq:27}. Then,
$$ \PP ( \| X \|_p^p \leq \kappa_{p,n}^p - \eps ) = \left( 1 - \frac{\eps}{\kappa_{p,n}^p} \right)^{n/p} \ge 1 - C \eps
$$
and therefore  $\text{Vol}(A_1)\ge 1-C\eps $.

We turn to bound the volume of $A_2$. To this end we claim that $\mathbb P ( RX_i\in [1,2] ) \le C/R$. Indeed, using the notation of Theorem~\ref{thm:asaf}, there exist some constant $C_0$ such that $\mathbb P ( |X_i| \ge C_0 |g_i| ) \le Ce^{-cn}$. Thus,
\begin{equation*}
    \mathbb P (RX_i\in [1,2]) \le Ce^{-cn} + \mathbb P ( |g_i| \le 2C_0/R ) \le C/R,
\end{equation*}
where the last inequality follows as the density of $g_i$ is bounded. Thus, by linearity of expectation $\mathbb E \big[ I(X) \big]\le Cn/R$. Finally, by Markov's inequality we have that $\mathbb P \big( I(X)\ge n/(R\eps ) \big) \le C\eps $ and therefore $\text{Vol} (A_2) \ge 1-C\eps $.

Lastly, we bound the volume of $A_3$. Recall the definition of $\tilde{X}_i$ in \eqref{eq: def of tilde X} and note that the density of $\tilde{X}_i$ is given by $e^{-|t|^p/a_n^p}/\big( 2a_n \Gamma(1+1/p) \big) $ where $a_n$ satisfies $1/6 \le a_n\le 1$. Recall also that $\psi $ is a smooth bump function supported on $[1,2]$ and that $\varphi (x)=r\psi (Rx)$. Using integration by parts twice and the fact that $g=(\varphi ^2/2)''$ we obtain
\begin{equation*}
\begin{split}
     \big| \mathbb E [g(\tilde{X}_i)] \big| &\le C \Big| \int _0^{\infty } (\varphi ^2)''(x)e^{-cx^p/a_n^p}dx  \Big|  \le C \Big| \int _0^\infty  (\varphi ^2)'(x) e^{-x^p/a_n^p} x^{p-1} dx \Big| \\
    &\le C \int _0^{\infty } \varphi ^2(x) x^{p-2} dx \le C R^{1-p}r^2\le C \eps ^2 /n,
\end{split}
\end{equation*}
where in the fourth inequality we used that $|\varphi (x)|\le r$ for all $x$ and that $\varphi$ supported on $[1/R,2/R]$ and in the last inequality we used  \eqref{eq:assumption}.  We now let $h:=\psi ''\psi +(\psi ')^2$ and note that $g(x)=R^2r^2h(Rx)$. By the first part of Claim~\ref{claim:independence phi} with the function $\varphi =h$ and with $\varphi =-h$ we have that $|\mathbb E [h(RX_i)-h(R\tilde{X}_i)]|\le C/n$ and therefore
\begin{equation}\label{eq:expected}
    \Big| \mathbb E \Big[ \sum _{i=1}^n g(X_i) \Big]  \Big| \le CR^2r^2+ \sum _{i=1}^n  \big| \mathbb E \big[ g(\tilde{X}_i) \big] \big|  \le  C\eps ^2.
\end{equation}

Next, using the second part of Claim~\ref{claim:independence phi} we obtain
\begin{equation}\label{eq:vari}
\begin{split}
    \text{Var} \Big(  \sum _{i=1}^n g(X_i) \Big) &= R^4r^4 \sum _{i=1}^n  \text{Var} \big( h(RX_i) \big) +R^4r^4\sum _{i\neq j} \text{Cov} \big( h(RX_i),h(RX_i) \big) \\
    &\le CnR^3r^4  \le  C\eps ^2.
\end{split}
\end{equation}
Finally, by \eqref{eq:expected}, \eqref{eq:vari} and Chebyshev's inequality
\begin{equation*}
    \mathbb P \Big(  \Big| \sum _{i=1}^n g(X_i ) \Big| \ge \eps  \Big) \le C\eps
\end{equation*}
and therefore $\text{vol} (A_3)\ge 1-C\eps $.
\end{proof}

We turn to prove Lemma~\ref{lem:density is close to 1 p>2}. To this end we need the following claims.

\begin{claim}\label{lem:11}
For all $y\in A$ we have that $\mathbb P \big( x(y,\delta)\notin B_p^n \big) \le \exp \big( -c\eps ^4 R^{-2}r^{-2}  \big) $.
\end{claim}

\begin{claim}\label{claim:fourth order}
We have that
\begin{equation*}
    \mathbb E \big[ \big(1+\varphi '(x_1)\delta _1 \big)^{-1} \big]=1+g(y_1)+O(R^4r^4) \quad \text{and} \quad  \mathbb E \big[ \big(1+\varphi '(x_1)\delta _1 \big)^{-2} \big]=1+O(R^2r^2).
\end{equation*}
\end{claim}

Using these claims we can easily prove Lemma~\ref{lem:density is close to 1 p>2}.

\begin{proof}[Proof of Lemma~\ref{lem:density is close to 1 p>2}]
Let $y\in A$ and recall that $I=I(y):=\{i\le n: 1\le Ry_i \le 2\}$ satisfies $|I|\le n/(R\eps )$. Note that $\varphi ' (x_i)=0$ for any $i\notin I$  and therefore the product in \eqref{eq:density p} can be written as a product over $i\in I$. Thus, by Lemma~\ref{lem:density} and Cauchy-Schwarz inequality
\begin{equation}\label{eq:123}
\begin{split}
    \Big| f(y)- \prod _{i\in I}&\mathbb E \big[ \big(1+\varphi '(x_i)\delta _i \big)^{-1} \big] \Big| = \mathbb E \Big[ \mathds 1 \{x(y,\delta)\notin B_p^n \} \cdot \prod _{i\in I} \big( 1+\varphi '( x_i)\delta _i \big)^{-1} \Big]\\
    &\le \sqrt{ \mathbb P \big( x(y,\delta)\notin B_p^n \big)} \cdot   \prod_{i\in I}  \sqrt{ \mathbb E \big[ \big(1+\varphi '(x_i)\delta _i \big)^{-2} \big]  } \\
    &\le \exp \big(-c\eps ^4 R^{-2}r^{-2}  +C|I|R^2r^2 \big) \le \exp \big(-c\eps ^4 R^{-2}r^{-2}  +C\eps ^{-1} nRr^2 \big),
\end{split}
\end{equation}
where in the second inequality we used Claim~\ref{lem:11} and the second part of Claim~\ref{claim:fourth order}. The right hand side of \eqref{eq:123} is at most $e^{-c \eps} \leq 1 - C \eps$ for sufficiently large $n$  by \eqref{eq:assumption}.

Moreover, using the first part of Claim~\ref{claim:fourth order} we have
\begin{equation}\label{eq:124}
\begin{split}
   \prod _{i\in I}\mathbb E &\big[ \big(1+\varphi '(x_i)\delta _i \big)^{-1} \big]= \prod _{i\in I} \Big( 1+g(y_i)+O(R^4r^4) \Big)\\
   &=\exp \Big( \sum _{i=1}^n g(y_i) +O(|I|R^4r^4) \Big) =\exp \big( O(\eps +\eps ^{-1} nR^3r^4 ) \big)=1+O(\eps ),
\end{split}
\end{equation}
where in the third equality we used that $y\in A_2\cap A_3$ and in the last equality we used \eqref{eq:assumption}. The lemma follows from \eqref{eq:123} and \eqref{eq:124} as long as $n$ is sufficiently large.
\end{proof}

We turn to prove Claim~\ref{lem:11} amd Claim~\ref{claim:fourth order}

\begin{proof}[Proof of Claim~\ref{lem:11}]
Let $y\in A$ and recall that $x=x(y,\delta)$ is defined by $x=(x_1,\dots ,x_n)$ where $x_i$ is the solution to the equation $y_i=x_i+\varphi (x_i)\delta _i$ . Define the random variable
\begin{equation*}
    S:=\sum _{i\in I} y_i^{p-1} \varphi (y_i) \delta _i.
\end{equation*}
Using a second order Taylor expansion we have almost surely
\begin{equation}\label{eq:21}
\begin{split}
    ||x||_p^p&=\sum _{i\notin I}|x_i|^p +\sum _{i\in I} \big( y_i- \varphi (x_i)\delta _i \big) ^p\\
    &=\sum _{i\notin I}|y_i|^p+ \sum _{i\in I } \big[ y_i^p-py_i^{p-1}\varphi (x_i)\delta _i+O\big( y_i^{p-2} r^2  \big) \big]  \\
    &= ||y||_p^p - p\sum _{i\in I } \big[  y_i^{p-1}\varphi (y_i)\delta _i+O\big( R^{2-p} r^2 \big) \big] \\
    &= ||y||_p^p + O \big( |S| + |I|R^{2-p} r^2 \big)\\
    &=||y||_p^p + O \big( |S| + \eps ^{-1}nR^{1-p} r^2 \big) =||y||_p^p + O \big( |S| + \eps ^2 \big),
\end{split}
\end{equation}
where in the third equality we used that $|\varphi (x_i)-\varphi (y_i)|\le CRr^2$, in the fifth equality we used that $y\in A_2$ and in the last inequality we used \eqref{eq:assumption}. We turn to bound the sum $S$ with high probability. The terms in this sum are almost surely bounded by $CR^{1-p}r $ and therefore by Azuma's inequality (see for example \cite[Theorem~7.4.2]{noga}) we have that
\begin{equation*}
    \mathbb P \big( |S|\ge \eps ^{3/2} \big) \le \exp \Big( \frac{-c\eps ^3}{ |I|R^{2-2p}r^2 } \Big) \le \exp \big(- c\eps ^4 n^{-1} R^{2p-1}r^{-2} \big) \le \exp \big( -c\eps ^4 R^{-2}r^{-2} \big),
\end{equation*}
where in the second inequality we used that $|I|=|I(y)|\le \frac{n}{R\eps }$ and in the last inequality we used \eqref{eq:assumption}. Substituting the last estimate into \eqref{eq:21} we get that
\begin{equation*}
    \mathbb P (x\notin B_p^n)=\mathbb P \big( ||x||_p^p > \kappa _{p,n}^p \big) \le \mathbb P \big( ||x||_p^p \ge ||y||_p^p +\eps  \big) \le \exp \big( -c\eps ^4 R^{-2}r^{-2} \big).
\end{equation*}
where the last inequality holds for a sufficiently small $\eps $. This finishes the proof of the claim~\ref{lem:11}.
\end{proof}

\begin{proof}[Proof of Claim~\ref{claim:fourth order}]
Recall that $\varphi (y)=r\psi (Ry)$ where $\psi $ is a fixed bump function and therefore $\varphi '(y)=O(Rr), \varphi ''(y)=O(R^2r)$ and $\varphi'''(y) = O(R^3 r)$. We have that
\begin{equation}\label{eq:23}
    x_1=y_1 -\varphi (x_1)\delta _1
\end{equation}
and therefore $x_1=y_1+O(r)$. Substituting this estimate into the right hand side of \eqref{eq:23} we get that $x_1=y_1-\varphi (y_1)\delta _1+O(Rr^2)$. Substituting the last estimate once again into the right hand side of \eqref{eq:23} we get
\begin{equation*}
    x_1=y_1-\delta _1\varphi (y_1)+\varphi '(y_1)\varphi (y_1)+O(R^2r^3).
\end{equation*}
Using the Taylor expansion of the function $\varphi '$ around $y_1$ we obtain
\begin{equation*}
    \varphi '(x_1)= \varphi '(y_1)-\delta _1\varphi ''(y_1)  \varphi (y_1)+ \varphi ''(y_1) \varphi '(y_1) \varphi (y_1)   +\frac{1}{2} \varphi '''(y_1)  \varphi (y_1)^2 +O(R^4r^4).
\end{equation*}
Thus, using the fourth order Taylor expansion of the function $1/(1+w)$ we obtain
\begin{equation*}
\begin{split}
\mathbb E \big[ \big( 1+\varphi '(x_1)\delta _1 \big)^{-1} \big]&= 1-\mathbb E \big[ \varphi '(x_1)\delta _1 \big] + \EE \big[ \varphi '(x_1)^2 \big] -\mathbb E \big[ \varphi '(x_1)^3\delta _1 \big] +O(R^4r^4) \\
&= \varphi ''(y_1)\varphi (y_1) +\varphi '(y_1)^2+O(R^4r^4)=g(y_1)+O(R^4r^4)
\end{split}
\end{equation*}
This finishes the proof of the first part of the claim. The second part follows using the same arguments.
\end{proof}

\subsection{Lower bound when $1<p<2$}\label{sec:lower when p<2}

In this section we prove the following proposition.

\begin{proposition}\label{prop:lower bound p<2}
 For all $1<p<2 $ we have that $L(B_p^n,1/2)=\Omega _p\big( (\log n)^{\frac{2-p}{2p}}  \big) $.
\end{proposition}

The proof is similar to the case $p>2$ but with one additional ingredient. In this case, perturbing each coordinate independently will push the random point outside of $B_p^n$ with high probability. To overcome this issue we perturb each pair of coordinates independently. Let $\psi :\mathbb R ^2 \to \mathbb R $ be a fixed, non-negative, smooth, two dimensional, bump function supported on $[1,2]^2$. Let $R_1,R_2 \geq 1, 0 < r < 1$ and let $\varphi (x_1,x_2):=r\psi (R_1(x_1-R_2), R_1(x_2-R_2))$. Finally, let
\begin{equation}\label{eq:def of f}
    h(x_1,x_2):=\varphi (x_1,x_2) \cdot \big(  x_1^{1-p},-x_2^{1-p} \big).
\end{equation}
The function $h$ will be the absolute value of the perturbation we apply to the coordinates $(x_{2i-1},x_{2i})$. As explained in the introduction, the idea in here is to perturb the first coordinate of the pair as much as possible and then use the second coordinate of the pair in order to ``correct" the change in the $p$ norm. The fact that this perturbation does not change the $p$ norm by much is apparent in equation (\ref{eq:9923}) below.

Next, define the random variable $Y=(Y_1\dots , Y_n)$ by
\begin{equation}\label{eq:def of Y 2}
 \big(  Y_{2i-1 },Y_{2i}\big):= \big( X_{2i-1 }, X_{2i} \big) +\delta _i h\big( X_{2i-1 }, X_{2i} \big) ,\quad i \le \lfloor n/2 \rfloor,
\end{equation}
where $\delta _i$ are i.i.d. $\!$symmetric $\{-1,1\}$ Bernoulli random variables and if $n$ is odd we let $Y_n=X_n$.

\begin{proposition}\label{prop:total variation 2}
For any $1<p<2$ there exists $\eps >0$ such that the following holds. Let $n\ge 1 $ sufficiently large and let $R_1,R_2,r>0$ such that
\begin{equation}\label{eq:assumption2}
    1\le R_2\le \log ^{1/p} n,\quad R_1=\log n, \quad nR_1^{-2} r^2 R_2^{-p} e^{-2R_2^p/a_n^p}\le \eps ^2,\quad r^5n^2 \le 1,
\end{equation}
where $a_n$ is given in \eqref{eq: def of tilde X}. Then, the random variable $Y$ given in \eqref{eq:def of Y 2} satisfies
\begin{equation*}
    d_{TV}( X,Y )\le 1/4.
\end{equation*}
\end{proposition}

We turn to prove Proposition~\ref{prop:lower bound p<2}. In the proof and throughout this section we use the notation $\tilde{O}$ to hide a poly-logarithmic factor of the form $\log ^Cn$ where, as usual, we allow the constant $C$ to depend on $p$. In order to simplify the arguments we also assume throughout the section that $n$ is even. The proof for odd $n$ is similar.

\begin{proof}[Proof of Proposition~\ref{prop:lower bound p<2}]
    Let $n\ge 1$ sufficiently large and even, $R_1:=\log n $ and $R_2=c_1\log ^{1/p}n$ where $c_1:=0.01$. Let $X$ be a uniform point in $B_p^n$ and define the random variable $W:=(W_1,\dots ,W_n)$ where for any $i\le \lfloor n/2 \rfloor $ we let
    \begin{equation*}
        (W_{2i-1},W_{2i}):= \psi \big( R_1(X_{2i-1}-R_2),R_2(X_{2i}-R_2) \big) \big( X_{2i-1}^{1-p},-X_{2i}^{1-p} \big) \delta _i.
    \end{equation*}
    We turn to show that $|W|$ is typically large. We have that $|W|^2=\sum _{i=1}^{ n/2 } \xi (X_{2i-1},X_{2i})$ where
    \begin{equation*}
        \xi (x_1,x_2) := \psi \big( R_1(x_{1}-R_2),R_2(x_{2}-R_2) \big)^2 \big( x_{2i-1}^{2-2p}+x_{2i}^{2-2p} \big).
    \end{equation*}

    Define the random variable $N:=\sum _{i=1}^{n/2}\xi (\tilde{X}_{2i-1},\tilde{X}_{2i})$.
    The function $\xi $ is supported on $[R_2+1/R_1,R_2+2/R_1]^2$ where the density of the pair $(\tilde{X}_1,\tilde{X}_2)$ is at least $ce^{-2R_2^{p}/a_n^p}=c n^{-2c_1^p/a_n^p}$. It follows that $\mathbb E \big[ \xi (\tilde{X}_{1},\tilde{X}_2)\big] \ge c R_1^{-2} R_2^{2-2p} n^{-2c_1^p/a_n^p} $ and therefore $\mathbb E [N]\ge c n R_1^{-2} R_2^{2-2p} n^{-2c_1^p/a_n^p} $. Note that $1/6\le a_n\le 1$ and therefore $\mathbb E [N]\ge n^{3/4}$. Next, since $\tilde{X}_i$ are independent we have $\text{Var}(N) = \tilde{O}(n)$. Thus, by Chebyshev's inequality, there exists some $c_2>0$ such that
    \begin{equation}\label{eq:848}
        \mathbb P \Big(  \sum _{i=1}^{n/2} \xi (\tilde{X}_{2i-1},\tilde{X}_{2i})  \ge c_2nR_1^{-2} R_2^{2-2p} n^{-2c_1^p/a_n^p} \Big) \ge 0.99,
    \end{equation}
    as long as $n$ is sufficiently large. In order to bound $|W|$ it suffices to replace the random variables $\tilde{X}_i$ with $X_i$ in the last estimate. To this end, note that the function $\xi $ and its partial derivatives are bounded by $\tilde{O}(1)$ and therefore by Claim~\ref{claim:X and tilde X} we have $\mathbb E \big|\xi (X_1,X_2)-\xi (\tilde{X}_1,\tilde{X}_2) \big| = \tilde{O} (n^{-1/2})$. Thus, by \eqref{eq:848} there exists $c_3>0$ such that
     \begin{equation}\label{eq:W is large 2}
     \begin{split}
        \mathbb P \Big( |W| \ge c_3\sqrt{n} &R_1^{-1} R_2^{1-p} n^{-c_1^p/a_n^p} \Big) \\
        &=  \mathbb P \Big(  \sum _{i=1}^{n/2} \xi ({X}_{2i-1},{X}_{2i})  \ge c_3^2nR_1^{-2} R_2^{2-2p} n^{-2c_1^p/a_n^p} \Big) \ge 0.99.
    \end{split}
    \end{equation}

    The rest of the proof is almost identical to the proof of Proposition~\ref{prop:lower bound p>2} and some of the details are omitted. We let $r_0:=\eps R_1R_2^{p/2} n^{c_1^p/a_n^p} n^{-1/2}$ and note that, by the choice of $c_1$ and the definition of $a_n$ in \eqref{eq: def of tilde X}, we have that $r_0^5n^2\le 1 $. By Proposition~\ref{prop:total variation 2}, for all $A\subseteq B_p^n$ and for all $0 < r < r_0$ we have $\mathbb P \big( X+rW\in A \big) \ge 1/4$. Thus, by \eqref{eq:W is large 2}, there exist $x,w\in \mathbb R ^n$ with $|w|\ge c_3\sqrt{n} R_1^{-1} R_2^{1-p} n^{-c_1^p/a_n^p}$ such that for $U\sim U[0,r_0]$ we have that $\mathbb P (x+Uw\in A)\ge 1/5$. It follows that the line $\ell $ containing $x$ and $x+w$ satisfies
    \begin{equation*}
        |\ell \cap A | \ge |r_0w|/5\ge c_{\eps } R_2^{1-p+p/2} \ge c_{\eps }(\log n) ^{\frac{2-p}{2p}}.
    \end{equation*}
    This finishes the proof of the proposition.
\end{proof}

As in the proof of Proposition~\ref{prop:total variation p}, we start by computing the density of $Y$. To this end, given $y\in \mathbb R ^n$, for each $i\le n/2$, let $x_{2i-1}$ and $x_{2i}$ be the random variables defined as the unique solutions to the equation
\begin{equation*}
    (y_{2i-1},y_{2i}) =(x_{2i-1},x_{2i}) +\delta _i h(x_{2i-1},x_{2i}).
\end{equation*}
When $r \le n^{-2/5}$ it is clear that the differential of the map $(x_{2i-1},x_{2i}) \mapsto (x_{2i-1},x_{2i}) +\delta _i h(x_{2i-1},x_{2i})$ is invertible and that the map is a diffeomorphism. We let $x=x(y,\delta)=(x_1,\dots ,x_n)$ and let $J(w_1,w_2,z)$ be the Jacobian determinant of the map $(w_1,w_2) \mapsto (w_{1},w_{2}) +z h(w_{1},w_{2})$ at the point $(w_1,w_2)$. The following lemma follows from the change of variables formula in the same way as Lemma~\ref{lem:density}.

\begin{lemma}\label{lem:density 2}
The density of the random variable $Y=(Y_1,\dots ,Y_n)$ defined in \eqref{eq:perturbation} is given  by
\begin{equation*}
    f(y)=\mathbb E \Big[ \mathds 1 \big\{ x(y,\delta)\in B_p^n \big\} \cdot \prod _{i=1}^{n/2}  J(x_{2i-1},x_{2i},\delta _i)^{-1} \Big],\quad y\in \mathbb R ^n.
\end{equation*}
\end{lemma}

As in the proof of Proposition~\ref{prop:total variation p}, in order to estimate the density given in Lemma~\ref{lem:density 2}, we restrict our attention to a set of almost full measure. To this end, define the function
\begin{equation}\label{eq:def of g}
    g(y_1,y_2):=\frac{1}{2} \frac{\partial ^2 h_1^2}{ \partial y_1^2}(y_1,y_2)  +\frac{\partial ^2 (h_1h_2)}{\partial y_1 \partial y_2}(y_1,y_2)+\frac{1}{2} \frac{\partial ^2 h_2^2}{ \partial y_2^2}(y_1,y_2), \quad (y_1,y_2)\in \mathbb R ^2
\end{equation}
and the set $A:=A_1\cap A_2$ where
\begin{equation*}
    A_1:=\big\{ y\in B_p^n: ||y||_p^p \le \kappa _{p,n}^p - \eps  \big\},\quad A_2:=\Big\{ y : \Big| \sum _{i=1}^{n/2} g(y_{2i-1},y_{2i}) \Big|  \le \eps  \  \Big\}.
\end{equation*}

Proposition~\ref{prop:total variation 2} clearly follows from the following two lemmas.

\begin{lemma}\label{lem:volume2}
We have that $\text{Vol} (A) \ge 1-C\eps $.
\end{lemma}

\begin{lemma}\label{lem:density is close to 1}
For any $y\in A$ we have that $|f(y)-1|\le C\eps $.
\end{lemma}

We start by proving Lemma~\ref{lem:volume2}.

\begin{proof}[Proof of Lemma~\ref{lem:volume2}]
 We have that $\text{Vol} (A_1) \ge 1-C\eps $ by the same arguments as in the proof of Lemma~\ref{lem:13}.

   We turn to bound the volume of $A_2$. To this end, we bound $ \mathbb E \big[ g(\tilde{X}_1,\tilde{X}_2) \big] $. Note that the density of $\tilde{X}_i$ at $y_i$ is given by $A_n\exp (-(|y_i|^p)/a_n^p)$ for some sequence $c\le A_n\le C$ and therefore
\begin{equation*}
    \mathbb E \big[ g(\tilde{X}_1,\tilde{X}_2) \big] =  A_n^2 \int _{\mathbb R^2 }  g(y_1,y_2) e^{-(y_1^p+y_2^p)/a_n^p} dy_1dy_2  = A_n^2\big(  I_1+I_2+I_3 \big)
\end{equation*}
where $I_1,I_2$ and $I_3$ are the tree integrals corresponding to the first, second and third terms in the right hand side of \eqref{eq:def of g}. Using integration by parts twice we obtain
\begin{equation*}
\begin{split}
    I_1&=\frac{1}{2} \int _{\mathbb R^2 } \frac{\partial ^2 h_1^2}{ \partial y_1^2} (y_1,y_2) e^{-(y_1^p+y_2^p)/a_n^p} dy_1dy_2
    =\int _{\mathbb R^2 } \frac{\partial h_1^2}{ \partial y_1} (y_1,y_2) \frac{py_1^{p-1}}{2a_n^p} e^{-(y_1^p+y_2^p)/a_n^p} dy_1dy_2 \\
    &= \int _{\mathbb R^2 }  h_1^2(y_1,y_2) \Big(\frac{p^2y_1^{2p-2}}{2a_n^{2p}} -\frac{p(p-1)y_1^{p-2}}{2a_n^p} \Big) e^{-(y_1^p+y_2^p)/a_n^p} dy_1dy_2.
\end{split}
\end{equation*}
Similarly we have that
\begin{equation*}
\begin{split}
    I_3= \int _{\mathbb R^2 }  h_2^2(y_1,y_2) \Big(\frac{p^2y_2^{2p-2}}{2a_n^{2p}} -\frac{p(p-1)y_2^{p-2}}{2a_n^p} \Big) e^{-(y_1^p+y_2^p)/a_n^p} dy_1dy_2
\end{split}
\end{equation*}
and
\begin{equation*}
\begin{split}
    I_2= \int _{\mathbb R^2 }  (h_1h_2)(y_1,y_2) \frac{p^2(y_1y_2)^{p-1}}{a_n^{2p}}  e^{-(y_1^p+y_2^p)/a_n^p} dy_1dy_2.
\end{split}
\end{equation*}
Adding these contributions we obtain
\begin{equation}\label{eq:g(X_1,X_2)}
\begin{split}
    &\big| \mathbb E [g(\tilde{X}_1,\tilde{X}_2)] \big| \le C\big| I_1+I_2+I_3 \big| \le C\int _{\mathbb R^2 } \Big( \big( h_1(y_1,y_2) y_1^{p-1}+h_2(y_1,y_2) y_2^{p-1} \big)^2 +\\
    &\quad \quad \quad \quad\quad \quad \quad \quad \quad \quad \quad \quad +h_1^2(y_1,y_2)y_1^{p-2}+h_2^2(y_1,y_2)y_2^{p-2}  \Big)e^{-(y_1^p+y_2^p)/a_n^p}dy_1dy_2\\
    &\ \ =C\int _{\mathbb R^2 }  \varphi ^2(y_1,y_2)(y_1^{-p}+y_2^{-p})  e^{-(y_1^p+y_2^p)/a_n^p}dy_1dy_2\le C   R_1^{-2} r^2 R_2^{-p} e^{-2R_2^p/a_n^p} \le C \eps ^2/n,
\end{split}
\end{equation}
where in the equality we substituted the definition of $h$ in \eqref{eq:def of f} and in the last inequality we used the assumption in \eqref{eq:assumption2}. Note the important cancellation of the first term in the integral. This cancellation is related to the fact that the perturbation given in \eqref{eq:def of Y 2} typically does not push the random point outside the ball.

\medskip

Next, let $N:= \sum _{i=1}^{n/2} g(\tilde{X}_{2i-1},\tilde{X}_{2i})$ and note that by \eqref{eq:g(X_1,X_2)} we have that $|\mathbb E [N]| \le C\eps ^2$. Moreover, using that $\tilde{X}_i$ are independent and \eqref{eq:assumption2} we obtain $\text{Var}(N)\le n\mathbb E \big[ g(\tilde{X}_1,\tilde{X}_2)^2 \big] =\tilde{O}(nr^4)=\tilde{O}(n^{-3/5})$. Thus, by Chebyshev's inequality, there exists $C_1>0$ such that
\begin{equation*}
    \mathbb P \Big( \Big| \sum _{i=1}^{n/2} g(\tilde{X}_{2i-1},\tilde{X}_{2i})  \Big| \le C_1 \eps ^2 \Big) \ge 1-\eps  ,
\end{equation*}
as long as $n$ is sufficiently large. In order to bound the volume of $A_2$ it suffices to replace $\tilde{X}_i$ with $X_i$ in the last estimate. To this end note that $g$ and its partial derivatives are bounded by $\tilde{O}(r^2)$ and therefore by Claim~\ref{claim:X and tilde X}, we have $\mathbb E \big|g(X_1,X_2)-g(\tilde{X}_1,\tilde{X}_2)  \big|=\tilde{O}(r^2n^{-1/2})$. Thus, there exists $C_2>0$ such that
\begin{equation*}
    \mathbb P \Big( \Big| \sum _{i=1}^{n/2} g({X}_{2i-1},{X}_{2i})  \Big| \le C_2 \eps ^2 \Big) \ge 1-2\eps.
\end{equation*}
It follows that $\text{Vol}(A_2)\ge 1-2\eps $ for a sufficiently small $\eps $ and a sufficiently large $n$ depending on $\eps $.
\end{proof}

We turn to prove Lemma~\ref{lem:density is close to 1}. To this end we need the following claim.

\begin{claim}\label{claim:computation2}
We have that
\begin{equation*}
    \mathbb E \big[ J(x_{1},x_{2},\delta _1)^{-1} \big] = 1+g(y_1,y_2) +\tilde{O}(r^3).
\end{equation*}
\end{claim}

\begin{proof}
We have that
\begin{equation*}
\begin{split}
    J&(x_1,x_2,\delta _1)=\Big( 1+ \delta _1\frac{\partial h_1}{\partial x_1} (x_1,x_2) \Big)\Big( 1+ \delta _1\frac{\partial h_2}{\partial x_2} (x_1,x_2) \Big) - \frac{\partial h_1}{\partial x_2} (x_1,x_2)\frac{\partial h_2}{\partial x_1} (x_1,x_2) \\
    &=1+\delta _1\Big( \frac{\partial h_1}{\partial x_1} (x_1,x_2)+\frac{\partial h_2}{\partial x_2} (x_1,x_2) \Big)+\frac{\partial h_1}{\partial x_1} (x_1,x_2)\frac{\partial h_2}{\partial x_2} (x_1,x_2) -  \frac{\partial h_1}{\partial x_2} (x_1,x_2)\frac{\partial h_2}{\partial x_1} (x_1,x_2).
\end{split}
\end{equation*}
Next, we replace the random points $x_1,x_2$ with the deterministic points $y_1,y_2$. Note that the terms in the brackets are of order $\tilde{O}(r)$ while the other terms are of order $\tilde{O}(r^2)$. We have that $y_i=x_i+\tilde{O}(r)$ and therefore, in the terms outside the brackets, $x_1,x_2$ can be replaced by $y_1,y_2$ without changing the overall expression by more than $\tilde{O}(r^3)$. For the terms inside the brackets we use the expansion
\begin{equation*}
    (x_1,x_2)=(y_1,y_2) -\delta _1 h(y_1,y_2) +\tilde{O} ( r^2 ).
\end{equation*}
We obtain that
\begin{equation*}
\begin{split}
    J(x_1,x_2,\delta _1)=1&+\delta _1\Big( \frac{\partial h_1}{\partial x_1} (y_1,y_2)-\delta _1 \frac{\partial ^2 h_1}{\partial x_1 ^2}(y_1,y_2) h_1(y_1,y_2)-\delta _1 \frac{\partial ^2 h_1}{\partial x_1 \partial x_2}(y_1,y_2) h_2(y_1,y_2) \\
    & +  \frac{\partial h_2}{\partial x_2} (y_1,y_2) -\delta _1 \frac{\partial ^2 h_2}{\partial x_2 \partial x_1}(y_1,y_2) h_1(y_1,y_2)-\delta _1 \frac{\partial ^2 h_2}{\partial x_2 ^2}(y_1,y_2) h_2(y_1,y_2) \Big)\\
    &+\frac{\partial h_1}{\partial x_1} (y_1,y_2)\frac{\partial h_2}{\partial x_2} (y_1,y_2) -  \frac{\partial h_1}{\partial x_2} (y_1,y_2)\frac{\partial h_2}{\partial x_1} (y_1,y_2)+\tilde{O}(r^3)\\
    =1&+\delta _1\Big( \frac{\partial h_1}{\partial x_1}
    +  \frac{\partial h_2}{\partial x_2}   \Big)- \frac{\partial ^2 h_1}{\partial x_1 ^2} h_1- \frac{\partial ^2 h_1}{\partial x_1 \partial x_2} h_2- \frac{\partial ^2 h_2}{\partial x_2 \partial x_1} h_1- \frac{\partial ^2 h_2}{\partial x_2 ^2} h_2 \\
    & + \frac{\partial h_1}{\partial x_1} \frac{\partial h_2}{\partial x_2} -
    \frac{\partial h_1}{\partial x_2} \frac{\partial h_2}{\partial x_1} +\tilde{O}(r^3).
\end{split}
\end{equation*}
Thus, using a second order Taylor expansion of $1/(1+x)$ we obtain
\begin{equation*}
\begin{split}
    \mathbb E \big[ J(x_1,x_2,\delta _1) ^{-1} \big] =1&+\Big( \frac{\partial h_1}{\partial x_1}
    +  \frac{\partial h_2}{\partial x_2}   \Big) ^2+ \frac{\partial ^2 h_1}{\partial x_1 ^2} h_1+ \frac{\partial ^2 h_1}{\partial x_1 \partial x_2} h_2+ \frac{\partial ^2 h_2}{\partial x_2 \partial x_1} h_1 \\
    &+ \frac{\partial ^2 h_2}{\partial x_2 ^2} h_2-\frac{\partial h_1}{\partial x_1} \frac{\partial h_2}{\partial x_2} +  \frac{\partial h_1}{\partial x_2} \frac{\partial h_2}{\partial x_1}+\tilde{O}(r^3)\\
    =1&+\Big( \frac{\partial h_1}{\partial x_1}\Big) ^2
    +  \Big( \frac{\partial h_2}{\partial x_2}   \Big) ^2+ \frac{\partial ^2 h_1}{\partial x_1 ^2} h_1+ \frac{\partial ^2 h_1}{\partial x_1 \partial x_2} h_2\\
    &+ \frac{\partial ^2 h_2}{\partial x_2 \partial x_1} h_1
    + \frac{\partial ^2 h_2}{\partial x_2 ^2} h_2 +\frac{\partial h_1}{\partial x_1} \frac{\partial h_2}{\partial x_2} +  \frac{\partial h_1}{\partial x_2} \frac{\partial h_2}{\partial x_1}+\tilde{O}(r^3)\\
    =1&+\frac{1}{2} \frac{\partial ^2 h_1^2}{ \partial x_1^2} +\frac{1}{2} \frac{\partial ^2 h_2^2}{ \partial x_2^2} +\frac{\partial ^2 h_1h_2}{\partial x_1 \partial x_2} +\tilde{O}(r^3)=1+g(y_1,y_2)+\tilde{O}(r^3).
\end{split}
\end{equation*}
This finishes the proof of the claim.
\end{proof}

\begin{proof}[Proof of Lemma~\ref{lem:density is close to 1} ]
       First, we claim that for all $y\in A$ we have that $x=x(y,\delta)\in B_p^n$ with probability one. Using a second order Taylor expansion we obtain
       \begin{equation}\label{eq:9923}
       \begin{split}
           ||y||_p^p&=\sum _{i=1}^{n}|y_i|^p= \sum _{i=1}^{n/2} \big| x_{2i-1}+\delta _ih_1(x_{2i-1},x_{2i})\big| ^p+\big| x_{2i}+\delta _ih_2(x_{2i-1},x_{2i})\big| ^p\\
           &=\sum _{i=1}^{n/2} \Big( |x_{2i-1}|^p+px_{2i-1}^{p-1}h_1(x_{2i-1},x_{2i})\delta _i +p(p-1)x_{2i-1}^{p-2}h_1(x_{2i-1},x_{2i})^2 \\
           &  \quad \quad\quad \ \  +|x_{2i}|^p+px_{2i}^{p-1}h_2(x_{2i-1},x_{2i})\delta _i+p(p-1)x_{2i}^{p-2}h_2(x_{2i-1},x_{2i})^2 +\tilde{O}(r^3)\Big)\\
           &=||x||_p^p+\tilde{O}(nr^3) +   p(p-1)\sum _{i=1}^{n/2} x_{2i-1}^{p-2}h_1(x_{2i-1},x_{2i})^2+x_{2i}^{p-2}h_2(x_{2i-1},x_{2i})^2 .
       \end{split}
       \end{equation}
       where in the last equality we used the definition of $h$ in \eqref{eq:def of f}. This cancellation of the linear term in the expansion is the reason we perturbed pairs of coordinates instead of individual coordinates. Thus, using that $y\in A_1$ we obtain
       \begin{equation*}
           ||x||_p^p \le ||y||_p^p +\tilde{O} ( nr^3) \le \kappa _{p,n}^p,
       \end{equation*}
      where the last inequality holds for sufficiently large $n$ as $r^5n^2\le 1$. It follows that $x\in B_p^n$. Thus, for $y\in A$ we have
       \begin{equation*}
           f(y)=\prod _{i=1}^{n/2} \mathbb E \big[ J(x_{2i-1},x_{2i},\delta _i)^{-1} \big]=\exp \Big( \tilde{O}(nr^3)+ \sum _{i=1}^{n/2} g(y_{2i-1},y_{2i}) \Big)  =1+O(\eps ),
       \end{equation*}
       where in the second equality we used Claim~\ref{claim:computation2} and that $r^5n^2 \le 1$.
\end{proof}

\subsection{Lower bound when $p=1$}

Consider the simplex $\Delta ^n:=B_1^n\cap \{x\in R^n: \forall i ,\ x_i\ge 0 \}$ and let $\mu _n $ be the uniform probability measure on $\Delta ^n$. It suffices to prove the following lemma

\begin{lemma}\label{lem:simplex}
We have that $L(\mu _n ,1/2)=\Omega (n^{1/4})$.
\end{lemma}

It follows from the lemma that $L(B_1^n,1/2)=\Omega (n^{1/4})$. Indeed, let $A\subseteq B_1^n$ with $\text{Vol}(A)\ge 1/2$. There exist a quadrant $Q_{\eps }:=\{x\in \mathbb R ^n : \forall i,\  \eps _ix_i\ge 0 \}$ of the space such that $\text{Vol}(Q_{\eps }\cap A)\ge 2^{-n-1}$. Without loss of generality suppose that this is the first quadrant (the quadrant corresponding to all coordinates are positive or $\eps =(1,\dots ,1)$). We have that $\mu _n(Q_{\eps }\cap A)\ge 1/2$ and therefore by Lemma~\ref{lem:simplex} there exists a line $\ell $ so that $|\ell \cap A |\ge | \ell \cap Q \cap A |\ge c n^{1/4}$. This shows that $L(B_1^n,1/2)=\Omega (n^{1/4})$.

For the proof of Lemma~\ref{lem:simplex} we need the following claim. To state the claim we let $g:=(g_1,\dots ,g_n)$ be an i.i.d. $\!\!$sequence of $\exp (1)$ random variables and let $\delta _1,\dots ,\delta _n$ be an i.i.d sequence of symmetric $\{-1,1\} $ Bernoulli random variables independent of $g$. We also let $\psi $ be a smooth, non-negative bump function supported on $[1,2]$  . For $r>0$ define the random variable $f=f^{(r)}:=(f_1,\dots ,f_n)$ where $f_i:=g_i+r\psi (g_i)\delta _i$.

\begin{claim}\label{claim:123}
There exists $\eps >0$ depending only on $\psi $ such that for all $r\le \eps n^{-1/4}$ we have $d_{TV}(f,g)\le 1/4$.
\end{claim}

The proof of Claim~\ref{claim:123} is left as an exercise to the reader. The proof is a minor modification of the claims in Section~\ref{sec_aa} and Section~\ref{sec_bb}. Note that in these sections assumption (ii) from Theorem~\ref{thm_1140} is not required.

\begin{proof}[Proof of Lemma~\ref{lem:simplex}]
Let $g_1,\dots ,g_n,Z$ be an i.i.d. $\!\!\!\!\!$ sequence of $\exp (1)$ random variables. It follows from Theorem~\ref{thm:asaf} that the vector $X:=(X_1, \dots ,X_n)$ defined by
\begin{equation*}
X_i:= \frac{(n!)^{1/n}g_i}{2\big( Z+\sum _{j=1}^n g_j \big)}.
\end{equation*}
is uniformly distributed in $\Delta ^n$. Fix $\eps >0$ such that Claim~\ref{claim:123} holds and let $r\le \eps n^{-1/4}$. Let $f^{(r)}$ be the random variable from Claim~\ref{claim:123} and suppose that the variables $\delta _i$ in the definition of $f^{(r)}$ are independent of $g_1,\dots ,g_n$ and $Z$. Define the random variable $Y=Y^{(r)}=(Y_1,\dots ,Y_n)$ by
\begin{equation*}
Y_i:=\frac{(n!)^{1/n}f_i}{2\big( Z+\sum _{j=1}^n f_j \big)}.
\end{equation*}
By Claim~\ref{claim:123} we have that $d_{TV}(X,Y^{(r)})\le 1/4$ and therefore for any subset $A\subseteq \Delta ^n$ with $\mu _n(A)\ge 1/2$ we have that
\begin{equation}\label{eq:844}
    \mathbb P \big( Y^{(r)} \in A \big) \ge 1/4.
\end{equation}

To simplify the notations we define the random variables
\begin{equation*}
P:=\frac{2\big( Z+\sum _{j=1}^n g_j \big)}{(n!)^{1/n}},\quad Q:=\frac{2\sum _{j=1}^n \psi (g_j)\delta _j}{(n!)^{1/n}},\quad  W_i:= \psi (g_i)\delta _i -\frac{Q}{P}g_i
\end{equation*}
and $W:=(W_1,\dots ,W_n)$.
We have that $X_i=g_i/P$ and a straightforward computation shows that
\begin{equation}\label{eq:846}
Y_i=\frac{g_i+r\psi (g_i)\delta _i}{P+rQ}=\frac{g_i}{P}+\frac{r}{P+rQ} \Big( \psi (g_i)\delta _i -\frac{Q}{P}g_i \Big)=X_i+\frac{r}{P+rQ}W_i.
\end{equation}
Next, since \eqref{eq:844} holds for any $r\le \eps n^{-1/4}$, it holds when replacing $r$ with $U\sim U[0,\eps n^{-1/4}]$ independent of all other random variables. Thus, rewriting \eqref{eq:844} using \eqref{eq:846} we obtain
\begin{equation*}
    \mathbb P \Big(  X+\frac{U}{P+UQ}W\in A \Big) \ge 1/4.
\end{equation*}

By the central limit theorem, Stirling's formula and the fact that $5<2e<6$ we have that
\begin{equation*}
    \mathbb P \big( 5<P<6 \big)\ge 0.99, \quad \mathbb P \big( |Q|\le C_1n^{-1/2} \big)\ge 0.99, \quad \mathbb P \big( |W|\ge c_1\sqrt{n} \big) \ge 0.97
\end{equation*}
for some $C_1,c_1>0$ and $n$ sufficiently large. Thus, there are $x,w\in \mathbb R ^n$ with $|w|\ge c_1\sqrt{n}$ and $p,q\in \mathbb R $ with $5<p<6$ and $|q|\le C_1 n^{-1/2}$ such that
\begin{equation*}
    \mathbb P \Big(  x+\frac{U}{p+Uq}w\in A \Big) \ge 1/5.
\end{equation*}

It is straightforward to check that the ratio between the densities of the random variables $U/(p+Uq)$ and $U/p$ tends to $1$ as $n$ tends to infinity and therefore for sufficiently large $n$ we have
\begin{equation*}
    \mathbb P \big(  x+(U/p)w\in A \big) \ge 1/6.
\end{equation*}
It follows that the line $\ell $ containing $x$ and $x+w$ satisfies
\begin{equation*}
    |\ell \cap A|\ge \eps n^{-1/4}|w|/(6p) \ge c_\eps n^{1/4}
\end{equation*}
as needed.
\end{proof}

\appendix

\section{Proof of Claim~\ref{claim:independence phi}} \label{sec:A}

\begin{proof}[Proof of Claim~\ref{claim:independence phi}]
First, note that by a straightforward calculation with the density of $g_k$ we have that
\begin{equation}\label{eq:24}
    \mathbb E [|g_k|^p]=1/p \quad \text{and} \quad \text{Var}(|g_k|^p)=1/p.
\end{equation}
Thus, the random variable
\begin{equation}\label{eq:25}
    N:= \frac{1}{n} \sum _{k=3}^n \Big( |g_k|^p-\frac{1}{p}\Big)
\end{equation}
is roughly normally distributed with variance of order $1/n$. In particular we have that $\mathbb E [N]=0$ and $\mathbb E [N^{2m}]\le C_m/n^m$ for all $m\in \mathbb N$. In the definition of $N$ we do not sum over $k=1,2$ in order to make it independent of $g_1$ and $g_2$. By \eqref{eq:24}, \eqref{eq:25} and Theorem~\ref{thm:asaf} we have that $X_1$ is well approximated by $\tilde{X}_1$. We claim that the following, more accurate approximation of $X_1$, holds. We have
\begin{equation}\label{eq:26}
X_1:= \tilde{X}_1-\tilde{X}_1N+ \tilde{X}_1H_1,
\end{equation}
where $H_1$ is some random variable with $\mathbb E [H_1^{m}] \le C_m/n^m$ for all $m\ge 1 $. Intuitively, \eqref{eq:26} says that $X_1$ can be written as $\tilde{X}_1$ plus a random variable with $0$ expectation of order $n^{-1/2}$ plus a random variable of order $n^{-1}$. The approximation in \eqref{eq:26} follows from Theorem~\ref{thm:asaf} and a second order Taylor expansion of the function $(1+x)^{-1/p}$. Indeed, by \eqref{eq:27} we have
\begin{equation*}
\begin{split}
    X_1&= \tilde{X}_1 \Big( 1+\frac{p}{n}Z+\frac{1}{n} \sum _{k=1}^n (p|g_k|^p-1) \Big)^{-1/p} \\
    &=\tilde{X}_1 \Big( 1+pN+ \frac{1}{n} \big( p|g_1|^p+p|g_2|^p+pZ-2 \big) \Big)^{-1/p}=\tilde{X}_1-\tilde{X}_1N+\tilde{X}_1H_1
\end{split}
\end{equation*}
where we define $H_1$ in such a way that the last equality holds. The fact that $\mathbb E [H_i^m] \le C_m/n^m$ follows from $\mathbb E [N^{2m}]\le C_m/n^m$.

\medskip Next, let $\varphi $ be a differentiable function supported on $[-C_0,C_0]$ such that $\varphi '$ is Lipschitz. For all $x,\tilde{x} \in \mathbb R $ we have that
 \begin{equation*}
     \varphi (Rx)=\varphi (R\tilde{x}) +R\varphi '(R\tilde{x})(x-\tilde{x})+O\big(\mathds 1 \{ \min (|x|,|\tilde{x}|) \le C_0/R \} \cdot  R^2(x-\tilde{x})^2 \big).
 \end{equation*}
 Substituting \eqref{eq:26} into the last estimate we obtain
 \begin{equation*}
     \varphi (RX_1)=\varphi (R\tilde{X}_1)-R\varphi '(R\tilde{X}_1) \tilde{X}_1 N +O\big( \mathds 1 _{\mathcal A_1} \big(  |\tilde{X}_1 H_1 | R +R^2\big( \tilde{X}_1N+\tilde{X}_1H_1\big) ^2 \big) \big)
 \end{equation*}
 where $\mathcal{A}_1:=\big\{ \min (|X_1|,|\tilde{X}_1|) \le C_0/R  \big\}$. Define the event $\mathcal B _1:= \big\{ |\tilde{X}_1| \le 2C_0/R \big\}$ and note that
\begin{equation*}
    \mathbb P (\mathcal A_1 \setminus \mathcal B_1) \le \mathbb P \big( |\tilde{X}_1| \ge 2 |X_1| \big)\le \mathbb P \bigg(  pZ+\sum _{k=1}^n (p|g_k|^p-1)  \ge n  \bigg) \le Ce^{-cn},
\end{equation*}
where the last inequality is by Bernstein's inequality and the fact that $|g_i|^p$ has exponential tails. We have
 \begin{align}\label{eq:28}
      \varphi (RX_1) & =\varphi (R\tilde{X}_1)-R\varphi '(R\tilde{X}_1) \tilde{X}_1N +O\big( \mathds 1 _{\mathcal B_1} \big(  H_1 +(N+H_1)^2 \big) +\mathds 1_{\mathcal A _1 \setminus \mathcal B_1} M_1 \big) \\ &
     =\varphi (R\tilde{X}_1)-R\varphi '(R\tilde{X}_1) \tilde{X}_1 N +O\big( \mathds 1 _{\mathcal B_1} F_1 +\mathds 1_{\mathcal A _1 \setminus  \mathcal B_1} M_1 \big)
\nonumber \end{align}
 where $M_1$ and $F_1$ are some random variables with $\mathbb E [M_1^m]\le C_m$ and $\mathbb E [F_1^m] \le C_mn^{-m}$ for all $m\in \mathbb N$.  Thus, using that $N$ is independent of $\tilde{X}_1$, $\mathbb E [N]=0$ and Cauchy-Schwarz inequality we obtain
 \begin{equation*}
     \mathbb E \big[ \varphi (RX_1) \big] = \mathbb E \big[ \varphi (R\tilde{X}_1) \big] +O(n^{-1}).
 \end{equation*}
 This finishes the proof of the first part of the claim.

 We turn to prove the second part. We clearly have that $\mathbb E [\varphi (R\tilde{X}_1) ]\le C/R$ and therefore
 \begin{equation}\label{eq:something}
 \begin{split}
     \mathbb E \big[ \varphi (RX_1) \big] ^2 =\mathbb E \big[ \varphi (R\tilde{X}_1) \big] ^2 +O(n^{-1}R^{-1}).
 \end{split}
 \end{equation}
 Moreover, by the same arguments as in \eqref{eq:28} we can write
 \begin{equation}\label{eq:29}
     \varphi (RX_2) =\varphi (R\tilde{X}_2)-R\varphi '(R\tilde{X}_2) \tilde{X}_2N +O\big( \mathds 1 _{\mathcal B_2} F_2 +\mathds 1_{\mathcal A _2 \setminus \mathcal B_2} M_2 \big)
\end{equation}
where $M_2$ and $F_2$ are some random variables with $\mathbb E[M_2^m] \le C_m$ and $\mathbb E [F_2^m] \le C_m/n^m$.

We now expand the $16$ terms in the product of the right hand sides of \eqref{eq:28} and \eqref{eq:29} to obtain
\begin{equation}\label{eq:871}
     \mathbb E \big[ \varphi (RX_1) \varphi (RX_2) \big]=\mathbb E \big[ \varphi (R\tilde{X}_1) \big] \mathbb E \big[ \varphi (R\tilde{X}_2) \big] +O(n^{-1}R^{-1}).
\end{equation}
Since $\tilde{X}_1$, $\tilde{X}_2$ and $N$ are independent we have
\begin{equation*}
    \mathbb E \big[ \varphi (R\tilde{X}_1)\varphi (R\tilde{X}_2) \big]=\mathbb E \big[ \varphi (R\tilde{X}_1) \big]\mathbb E \big[ \varphi (R\tilde{X}_2) \big] \quad \text{and} \quad  \mathbb E \big[  R\varphi '(R\tilde{X}_1) \tilde{X}_1 N \varphi (R\tilde{X}_2) \big]=0.
\end{equation*}
Next, by Cauchy-Schwarz we have
\begin{equation*}
\begin{split}
    &\mathbb E \big[ \mathds 1 _{\mathcal B_1} F_1 \varphi (R\tilde{X}_2) \big] \le C\big(  \mathbb E [F_1^2]  \cdot \mathbb P (\mathcal B _1 \cap \mathcal B_2) \big)^{1/2}\le C/(nR)\\
    &\mathbb E \big[ \mathds 1_{\mathcal A _1 \setminus  \mathcal B_1} M_1 \varphi (R\tilde{X}_2) \big] \le C\big(  \mathbb E [M_1^2]  \cdot \mathbb P (\mathcal A _1 \setminus  \mathcal B_1) \big)^{1/2}\le Ce^{-cn}\\
    &\mathbb E \big[ \mathds 1 _{\mathcal B_1} F_1 R\varphi '(R\tilde{X}_2) \tilde{X}_2N \big] \le C\big(  \mathbb E [F_1^2N^2]  \cdot \mathbb P (\mathcal B _1 \cap \mathcal B_2) \big)^{1/2}\le C/(nR)\\
    &\mathbb E \big[ \mathds 1_{\mathcal A _1 \setminus  \mathcal B_1} M_1 R\varphi '(R\tilde{X}_2) \tilde{X}_2N \big] \le C\big(  \mathbb E [M_1^2N^2]  \cdot \mathbb P (\mathcal A _1 \setminus  \mathcal B_1) \big)^{1/2}\le C/(nR).
\end{split}
\end{equation*}
The other terms in the product are either clearly small by Cauchy-Schwarz or symmetric to one of the above terms. This finishes the proof of \eqref{eq:871}. The second part of the claim follows from \eqref{eq:something} and \eqref{eq:871}. Indeed, $\text{Cov}(\varphi (RX_1),\varphi (RX_2))=\mathbb E \big[ \varphi (RX_1) \varphi (RX_2) \big]-\mathbb E \big[ \varphi (RX_1) \big] ^2 =O(n^{-1}R^{-1})$.
\end{proof}

The proof of Claim~\ref{claim:X and tilde X} follows from (\ref{eq:26}) above.

\end{document}